\documentclass[twoside,11pt,reqno]{amsart}
\usepackage{amsmath,amssymb,amscd,mathrsfs,epic}

\makeatletter

\@addtoreset{equation}{section}

\newtheorem{Proposition}{Proposition}[section]
\newtheorem{Lemma}[Proposition]{Lemma}
\newtheorem{Theorem}[Proposition]{Theorem}
\newtheorem{Corollary}[Proposition]{Corollary}
\newtheorem{Definition}[Proposition]{Definition}
\newtheorem{Remark}[Proposition]{Remark}

\newtheorem{Example}[Proposition]{Example}
\newtheorem{Problem}[Proposition]{Problem}

\newcommand{\Hom}{\operatorname{Hom}}

\newcommand{\End}{\operatorname{End}}

\newcommand{\im}{\operatorname{im}}

\newcommand{\res}{\operatorname{res}}

\newcommand{\Stab}{\operatorname{Stab}}

\newcommand{\ga}{\gamma}
\newcommand{\Ga}{\Gamma}
\newcommand{\la}{\lambda}
\newcommand{\La}{\Lambda}
\newcommand{\al}{\alpha}
\newcommand{\be}{\beta}
\newcommand{\ep}{\epsilon}

\newcommand{\si}{\sigma}

\newcommand{\de}{\delta}
\newcommand{\De}{\Delta}

\newcommand{\Ker}{\operatorname{Ker}}

\newcommand{\Irr}{{\mathrm {Irr}}}

\renewcommand{\Im}{{\mathrm {Im}}}

\newcommand{\diag}{{\mathrm {diag}}}
\newcommand{\CC}{{\mathbb C}}
\newcommand{\CL}{{\mathcal C}}
\newcommand{\SCL}{{\mathcal S}}

\newcommand{\ZZ}{{\mathbb Z}}
\newcommand{\NN}{{\mathbb N}}

\newcommand{\NC}{{\mathcal N}}
\newcommand{\MC}{{\mathcal M}}

\newcommand{\AC}{{\mathcal A}}
\newcommand{\BC}{{\mathcal B}}
\newcommand{\GC}{{\mathcal G}}
\newcommand{\EC}{{\mathcal E}}
\newcommand{\GD}{G^{*}}
\newcommand{\FF}{{\mathbb F}}
\newcommand{\FQ}{\FF_{q}}
\newcommand{\SSS}{{\sf S}}
\newcommand{\AAA}{{\sf A}}

\newcommand{\LC}{{L_{\CC}}}

\newcommand{\VC}{{V_{\CC}}}
\newcommand{\WN}{{W_{\CC}}}
\newcommand{\dvg}{{d_{V}(g)}}
\newcommand{\dar}{{\downarrow}}
\newcommand{\uar}{{\uparrow}}
\newcommand{\ml}{{\mathfrak{m}_{\ell}}}
\newcommand{\mc}{{\mathfrak{m}_{\CC}}}
\newcommand{\dl}{{\mathfrak{d}_{\ell}}}

\newcommand{\tn}{\hspace{0.5mm}^{t}\hspace*{-0.2mm}}

\def\skipa{ \vspace{-1.1mm} & \vspace{-1.1mm} & \vspace{-1.1mm}\\}
\def\skipb{\vspace{-1.3mm} & \vspace{-1.3mm} & \vspace{-1.3mm}\\}
\renewcommand{\mod}{\bmod \,}

\newcommand{\sfr}{{\mathfrak{s}}}

\newcommand{\OLA}{{O_{\ell'}}}
\newcommand{\OLB}{{O_{\ell}}}
\newcommand{\FQT}{\FF_{q}^{\times}}
\newcommand{\IBr}{{\mathrm {IBr}}}

\begin{document}

\title[Representations of the general linear groups]{{\bf Representations of the general linear \\groups  which are
irreducible\\
over subgroups}}

\author{\sc Alexander S. Kleshchev}
\address
{Department of Mathematics\\ University of Oregon\\
Eugene\\ OR~97403, USA}
\email{klesh@uoregon.edu}

\author{\sc Pham Huu Tiep}
\address
{Department of Mathematics\\ University of Arizona\\
Tucson\\ AZ~85721, USA}
\email{tiep@math.arizona.edu}

\thanks{
\indent
Research supported by the NSF (grants DMS-0654147 and DMS-0600967). The paper has been written at MSRI. We thank the institute for support}


\maketitle

\section{Introduction}
\subsection{Aschbacher-Scott Program}
Primitive permutation groups have been studied since Galois and Jordan and have applications in many different areas. If $\Gamma$ is a transitive permutation group with a point stabilizer $M$ then $\Gamma$ is primitive if and only if $M<\Gamma$ is a maximal subgroup. So studying primitive permutation groups is equivalent to studying maximal subgroups.

In most problems
involving a finite primitive group $\Gamma$, the Aschbacher-O'Nan-Scott theorem \cite{AS} allows one to concentrate on the 
case where $\Gamma$ is almost quasi-simple, i.e. $L\lhd \Gamma/Z(\Gamma)\leq \operatorname{Aut}(L)$ for a non-abelian simple group $L$. The results of 
Liebeck-Praeger-Saxl \cite{LPS} and Liebeck-Seitz \cite{LiS} then allow one to 
assume furthermore that $\Gamma$ is a finite classical group. 

At this stage, 
Aschbacher's theorem \cite{A} severely restricts possibilities for the maximal subgroup $H$. Namely, if $H<\Gamma$ is maximal then 
\begin{equation}\label{EM}
\textstyle H \in \SCL\cup \bigcup^{8}_{i=1}\CL_{i},
\end{equation}
where $\CL_{i}$, $i = 1, \ldots ,8$, are
collections of certain explicit natural subgroups of $\Gamma$, and $\SCL$ is the collection
of almost quasi-simple groups that act absolutely irreducibly on the natural module for the classical group $\Gamma$. 

The converse however does not have to be true. 
So, to understand primitive permutation groups, one needs to determine
whether a subgroup $H$ in (\ref{EM}) is actually maximal in
$\Gamma$. For $H \in \cup^{8}_{i=1}\CL_{i}$, this has been done by
Kleidman-Liebeck \cite{KL}. 
Let $H \in \SCL$. If $H$ is {\em not} maximal then $H < G < \Gamma$ for a certain maximal subgroup  $G$ in $\Gamma$. The most challenging case to handle is when $G \in \SCL$ as well. This motivates the following problem. Throughout the paper $\FF$ is an algebraically closed field of characteristic $\ell\geq 0$. 

\begin{Problem}\label{restr}
Classify all triples $(G,V,H)$ where 
$G$ is an almost quasi-simple finite group, $V$ 
is an $\FF G$-module of dimension greater than one, and $H$ is a proper subgroup of $G$ such that 
the restriction $V\dar_{H}$ is irreducible.
\end{Problem}

For the purposes of Aschbacher-Scott program, it suffices to solve Problem \ref{restr} in the 
case where $H$ is almost quasi-simple. However, in some interesting situations 
it is possible to handle Problem \ref{restr} in full generality. We believe that this might be of importance for other applications in group theory.

For example, under the assumption $\ell>3$, Problem \ref{restr} has been solved for $G$ of alternating type, i.e. $G=\AAA_n,\SSS_n,\hat \AAA_n$ or $\hat \SSS_n$, see \cite{BK2,KS2,KT}, and  partial results are available even in the cases $\ell=3$ and $2$, see e.g. \cite{KS1}. Let us denote by $\operatorname{Lie}(p)$ the family of finite groups 
of Lie type over fields of characteristic~$p$. 
In the {\em defining characteristic case}, i.e. when $G \in \operatorname{Lie}(\ell)$, definitive results on Problem \ref{restr}  have been obtained by Liebeck, Seitz, and Testerman 
\cite{Lieb1}, \cite{S1}, \cite{STe}. Consider the {\it cross-characteristic case},
i.e.  $G \in \operatorname{Lie}(p)$ with $p \neq \ell$. When $G$ is a 
classical group, Seitz \cite{S2} lists possible 
candidates for the subgroup $H$ arising in Problem \ref{restr} under the condition 
that $H \in \operatorname{Lie}(p)$.

Throughout the paper $q$ is a prime power not divisible by the characteristic $\ell$ of the ground field $\FF$.  
Our main result is a solution of Problem~\ref{restr} in 
the case where $SL_n(q)\leq G\leq GL_n(q)$.

\subsection{Notation}\label{SStat}
For $\si\in\bar{\FF}_q^\times$ we denote by $[\si]$ the set of all roots of the minimal polynomial of 
$\si$ over $\FF_q$; in particular, $\#[\si] = \deg(\si)$. 
We say that $\si_1$ and $\si_2$ are {\em conjugate} if $[\si_1]=[\si_2]$. 
The {\em order} of $\si$, denoted $|\si|$, 
is its multiplicative order, and $\si$ is an $\ell$- (resp. $\ell'$-) element if $|\si|$ is a 
power of $\ell$ (resp. prime to $\ell$). If $\ell=0$, all elements are $\ell'$-elements.

We state the James' classification of irreducible $\FF GL_n(q)$-modules. An {\em $n$-admissible tuple} is a tuple 
\begin{equation}\label{EAT}
\big(([\si_1],\la^{(1)}),\dots,([\si_a],\la^{(a)})\big)
\end{equation}
 of pairs, where $\si_1,\dots,\si_a\in\bar{\FF}_q^\times$ are $\ell'$-elements, and 
$\la^{(1)},\dots,\la^{(a)}$ are partitions, such that $[\si_i]\neq[\si_j]$ for all 
$i\neq j$ and $\sum_{i=1}^a \deg(\si_i) \cdot |\la^{(i)}|=n$. 
An equivalence class of the $n$-admissible tuple (\ref{EAT}) up to a permutation of the 
pairs $([\si_1],\la^{(1)}),\dots,([\si_a],\la^{(a)})$ is called an {\em $n$-admissible symbol} and denoted 
\begin{equation}\label{EAS}
\sfr 
= \big[([\si_1],\la^{(1)}),\dots,([\si_a],\la^{(a)})\big]
\end{equation}
The set ${\mathfrak L}$ of $n$-admissible symbols is the labeling set 
for irreducible $\FF GL_n(q)$-modules. 
The module corresponding to the symbol (\ref{EAS}) is 
$$L({\mathfrak s}) = 
  L(\si_{1},\la^{(1)}) \circ \ldots \circ L(\si_{a},\la^{(a)})$$ 
(`$\circ$' denotes Harish-Chandra induction), and 
$
\{L({\mathfrak s})\mid {\mathfrak s}\in{\mathfrak L}\}
$ 
is a complete set of representatives of irreducible $\FF GL_n(q)$-modules. 
Note that the irreducible modules $L(\si,\la)$ are defined even when $\si$ is an $\ell$-singular (i.e. not $\ell'$-) element \cite{J, BDK}. Although such $L(\si,\la)$ can be expressed in terms of  admissible symbols and in that sense are redundant, it is convenient to use them. So we will not assume that $\si$ in $L(\si,\la)$ is an $\ell'$-element, 
unless otherwise stated.

For finite groups $H\leq G$ and an irreducible $\FF G$-module $W$, we denote by 
$\kappa^G_H(W)$ the number of composition factors (counting with multiplicities) of the restriction $W\dar_H$. 
In this paper we will work with the group $G$ such that $SL_n(q)\leq G\leq GL_n(q)$. 
If $W$ is an irreducible $\FF GL_n(q)$-module, then the restriction $W\dar_G$ is completely reducible and multiplicity free, with the number $\kappa_G(W):=\kappa^{GL_n(q)}_G(W)$ explicitly known, see (\ref{EB2}). Moreover, every 
irreducible $\FF G$-module $V$ appears in some $W\dar_G$. It will turn out (somewhat miraculously) that if $V'$ is another irreducible constituent of $W\dar_G$ and $H<G$, then $V\dar _H$ is irreducible if and only if $V'\dar_H$ is irreducible. So we will be able to  state our theorems solely in terms of the $GL_n(q)$-module $W$.

Parabolic subgroups $P$ which are stabilizers of a $1$-space or an $(n-1)$-space in the natural module $\FF_q^n$ will play an important role in the paper. Let  
$L:=GL_{n-1}(q)\times GL_1(q)$ be the natural block-diagonal subgroup of $GL_n(q)$, and let $P=UL$ be a parabolic subgroup of $GL_n(q)$ with the unipotent radical $U$ and Levi subgroup $L$. Note that $[L,L]\cong SL_{n-1}(q)$. 
The following technical notation will be used for the class of subgroups $H$ with $[P,P]\leq H\leq P$. For such $H$ we can write $H=UM$ for some $M\leq L$. Then we define the subgroup $M_1<GL_{n-1}(q)$ so that $M_1\times Z(GL_n(q))=MZ(GL_n(q))$ (we have identified $GL_{n-1}(q)=GL_{n-1}(q)\times\{1\}\leq L$). For $\si\in \bar{\FF}_q^\times$ such that $d:=\deg(\si)$ divides $n$ set
$$\kappa(\si,H):=\kappa^{GL_n(q)}_{M_1SL_n(q)}(L(\si,(n/d))).$$

\subsection{Statements of the Main Result}
We state the result under the assumption $n\geq 5$. Small rank cases are treated in $\S$\ref{SSR}. 
 If $SL_n(q)\leq H\leq G\leq GL_n(q)$ and $V$ is an irreducible $\FF G$-module then the number of irreducible constituents in $V\dar_H$ is known, see (\ref{EB1}),(\ref{EB2}) below. So in the theorem we may assume that $H$ does not contain $SL_n$. The subgroups $H<G$ are given {\em up to $G$-conjugacy} (a $G$-module is irreducible on restriction to $H$ if and only if it is irreducible on restriction to a conjugate of $H$).

\begin{Theorem}\label{main1}
Let $n\geq 5$, $SL_n(q) \leq G \leq GL_n(q)$, $H<G$ be a proper subgroup not containing $SL_n(q)$, and $V$ be an  
irreducible $\FF G$-module of dimension greater than $1$. Let $W$ be an irreducible $\FF GL_n(q)$-module such that $V$ is an irreducible constituent of $W\dar_G$. Then  $V\dar_{H}$ is irreducible if and only if 
one of the following holds:

\begin{enumerate}
\item[{\rm (i)}] $H \leq P$, where $P=UL$ is the stabilizer 
in $GL_n(q)$ of a $1$-space or an $(n-1)$-space in the natural $GL_n(q)$-module $\FQ^{n}$, $W = L(\sigma,(k))$ for some $\sigma\in\bar{\FF}_q^\times$ of degree $n/k > 1$, and one of the following holds:
\begin{enumerate}
\item[{\rm (a)}] $H\geq [P,P]$ and $\kappa(\si,H)=\kappa_G(W)$. 
\item[{\rm (b)}] $G=SL_n(3)$,  
$\si^2=-1$ if $\ell\neq 2$, and $[L,L]\leq H\leq L$. 
\item[{\rm (c)}] $G=SL_n(2)$, 
$\si\neq1=\si^3$ and $H= L$. 
\end{enumerate}

\item[{\rm (ii)}] $n$ is even, $W = L(\sigma,(1)) \circ L(\tau,(n-1))$ for some $\ell'$-elements 
$\sigma, \tau \in \FF_{q}^{\times}$ with $\tau \neq \sigma$ (in particular, $V = W\dar_{G}$ is a Weil representation of dimension 
$(q^{n}-1)/(q-1)$), and one of the following holds:
\begin{enumerate}
\item[{\rm (a)}] $Sp_{n}(q)Z(H) < H \leq CSp_{n}(q)$.
\item[{\rm (b)}] $H=Sp_n(q)Z(H)$ and $\tau \neq \pm \sigma$.
\item[{\rm (c)}] $2|q$, $n=6$, and $G_{2}(q)' \lhd H \leq GL_n(q)$.
\end{enumerate}
\end{enumerate}
\end{Theorem}

Note from Theorem~\ref{main1} that the subgroups $H$ with $V{\dar}_H$ irreducible are all almost quasi-simple, except for the ones showing up in (i)(a).  

Some key tools used in the proofs are as follows: Reduction Lemma~\ref{hom1} whose   idea goes back to Jan Saxl \cite{Saxl}; Hering's theorem on groups transitive on $1$-spaces; quantum group methods of \cite{BDK}; description of the number of composition factors in $W\dar_{SL_n(q)}$ for a $GL_n(q)$-module $W$ obtained in \cite{KTSL}; information on minimal polynomials of semisimple elements from \cite{TZ}; lower bounds from \cite{GT1,BK}; and, of course, Dipper-James theory.

\section{Preliminaries}\label{SGen}
\subsection{More notation}\label{SMoreNot}
Throughout the paper all groups are assumed to be finite. 
The  following notation is 
in addition to the one introduced in $\S$\ref{SStat}:
\begin{enumerate}
\item[$\la\vdash k$] means that $\la$ is a partition of $k$;
\item[$\la'$] denotes the transposed partition of $\la$;
\item[$H^g$] denotes $g^{-1}Hg$ for $H\leq G$ and $g\in G$;
\item[$\kappa^G_H(W)$] the number of composition factors in $W\dar_H^G$ (counting with multiplicities);
\item[$\ml(H)$] the largest degree of irreducible representations
of a group $H$ over~$\FF$;
\item[$\mc(H)$] the largest degree of irreducible representations
of a group $H$ over~$\CC$; 
\item[$\dl(H)$] the smallest degree $> 1$ of irreducible
projective representations of a group $H$ over~$\FF$; 
\item[$\IBr(H)$]  the set of isomorphism classes of irreducible $\FF H$-modules or the set of irreducible $\ell$-Brauer characters, depending on the context;
\item[$\Irr(H)$]  the set of isomorphism classes of irreducible $\CC H$-modules or the set of irreducible complex characters, depending on the context;
\item[$q=p^f$] power of a prime number $p$ such that $\ell{\not{|}}\,q$; 
\item[$s=s(d)$] the minimal positive integer such
that $q^{ds}\equiv 1\pmod{\ell}$; by convention, $s=\infty$ if $\ell=0$; here $d\in \NN$ usually denotes the degree of a fixed element $\si\in\bar\FF_q^\times$ which is clear from the context;
\item[$e=e(d)$] the smallest positive
integer  such that in $\FF$ we have 
$
\sum_{i=0}^{e-1}q^{di}=0
$;
\item[$GL_n$] denotes $GL_n(q)$; 
\item[$SL_n$]  denotes $SL_n(q)$;
\item[$CSp_{2n}(q)$] is the conformal symplectic group, i.e. the group of all invertible linear transformations of $\FF_q^{2n}$ which preserve up to a scalar a non-degenerate symplectic form. 
\item[$GU_{n}(q)$] is the group of all invertible linear transformations of $\FF_{q^2}^n$ which preserve a non-degenerate Hermitian form. 
\item[$CU_{n}(q)$] is the group of all invertible linear transformations of $\FF_{q^2}^n$ which preserve up to a scalar a non-degenerate Hermitian form. 
\item[$\NC$] $=\NC_n$ the natural $\FF_qGL_n$-module $\FQ^{n}$;
\item[$\Gamma L_{m}(q^{k})$] the group of all invertible $\FF_{q^k}$-semilinear (and $\FF_q$-linear)  
transformations of $(\FF_{q^{k}})^m$, considered naturally as a subgroup of $GL_{mk}(q)$.
\item[$AGL_n$] $=AGL_n(q)=\NC_n\rtimes GL_n$ the affine general linear group. We
usually consider $AGL_n$ as the natural subgroup of $GL_{n+1}$ consisting of all matrices with the last row equal to $(0\,0\,\dots\,0\,1)$;
\item[$P_1$] is the stabilizer in $GL_n$ of a $1$-space in $\NC$; 
\item[$P_{n-1}$] is the stabilizer in $GL_n$ of an $(n-1)$-space in $\NC$.
\end{enumerate}
If $X \leq Y$ are groups, $W$ is an
$\FF X$-module and $V$ is an $\FF Y$-module, then $V\dar_{X}$ is the restriction to $X$ of
$V$, and $W\uar^{Y}$ is the induction to $Y$ of $W$.

\begin{Lemma}\label{LPPD}{\rm \cite{Zs}} 
Let $m,n \in\ZZ_{\geq 2}$ and exclude the cases where $(m,n)=(2,6)$ or $m$ is a Mersenne prime and $n=2$. Then there is a prime $r> n$ that divides $m^n-1$ but not 
$\prod^{n-1}_{i=1}(m^{i}-1)$. 
\end{Lemma}

The prime $r$ as in the lemma is referred to as  a {\it primitive prime divisor} (p.p.d. for short) for $(m,n)$.

\subsection{Some Clifford theory}
For a partition $\la=(\la_1,\la_2,\dots)$ we denote $\Delta(\la):=\gcd(\la_1,\la_2,\dots)$. 
For a multipartition $\underline{\la}=(\la^{(1)},\la^{(2)},\dots,\la^{(a)})$ (which 
means that each $\la^{(i)}$ is a partition), $\underline{\la}'$ is the transposed 
multipartition $((\la^{(1)})',\dots,(\la^{(a)})')$. Set
$$\Delta(\underline{\la}):=\gcd(\Delta(\la^{(1)}),\dots,\Delta(\la^{(a)})).$$

The subgroup  $\OLA(\FF_q^\times)$  acts on the set ${\mathfrak L}$ of 
$n$-admissible symbols via
$$
\tau\cdot\big[([\si_1],\la^{(1)}),\dots,([\si_a],\la^{(a)})\big]=
  \big[([\tau\si_1],\la^{(1)}),\dots,([\tau\si_a],\la^{(a)})\big]
$$
for $\tau\in \OLA(\FF_q^\times)$.  
The order of the stabilizer group in  $\OLA(\FF_q^\times)$ of a symbol 
${\mathfrak s}\in {\mathfrak L}$ is called the {\em $\ell'$-branching number}
 of ${\mathfrak s}$ and is denoted 
$\kappa_{\ell'}({\mathfrak s})$. Next, if ${\mathfrak s}$ is of the form (\ref{EAS}), 
then the {\em $\ell$-branching number}  
$\kappa_{\ell}({\mathfrak s})$ is the $\ell$-part of 
$\gcd\big(n,q-1,\Delta(\underline{\la}')\big)=\gcd\big(q-1,\Delta(\underline{\la}')\big)$, 
where $\underline{\la}=(\la^{(1)},\la^{(2)},\dots,\la^{(a)})$. 

We will need several  results on branching numbers $\kappa^G_H(V)$ from \cite{KTSL}.

\begin{Lemma}\label{Llink2} {\rm \cite[Lemma~3.1]{KTSL}} 
Let $S\leq H$ be normal subgroups of a group $G$, $V\in\IBr(G)$, and $U$ be an irreducible component of $V\dar_H$. Then $\kappa^{G}_{S}(V)=\kappa^{G}_{H}(V)\cdot \kappa^{H}_{S}(U)$.
\end{Lemma}

\begin{Lemma}\label{linear} {\rm \cite[Lemma~3.2]{KTSL}} 
 Let $S\lhd G$ with $G/S$ a cyclic $\ell'$-group, and  
$V \in \IBr(G)$. Then $\kappa^G_S(V) = \sharp\{L \in \IBr(G/S) \mid V \cong V \otimes L\}$.
\end{Lemma}

\begin{Lemma}\label{link2} {\rm \cite[Corollary 3.5]{KTSL}} 
Let $S\lhd G$ with $G/S$ cyclic, $S\leq H\leq B\leq  G$ such that $B/H = \OLB(G/H)$, 
and $V \in \IBr(G)$. 
Then
$$\kappa^{G}_{H}(V) = \kappa^{G}_{B}(V) \cdot \min\{ (\kappa^{G}_{S}(V))_{\ell}, |G/H|_{\ell}\}.$$
\end{Lemma}

\begin{Lemma}\label{tensor} {\rm \cite[Lemma~4.1]{KTSL}} 
Let $\mathfrak{s}$ be an $n$-admissible symbol and $\tau\in \OLA(\FQT )$. Then 
$L(\mathfrak{s}) \otimes L(\tau,(n)) \cong L(\tau\cdot\mathfrak{s}).$
\end{Lemma}

\begin{Theorem}\label{TMainSL} {\rm \cite[Theorem~1.1]{KTSL}} 
Let $V = L({\mathfrak s})$ be 
the irreducible $\FF GL_n(q)$-module corresponding to  
${\mathfrak s}\in {\mathfrak L}$. Then $V\dar_{SL_n(q)}$ is a sum of 
 $\kappa_{\ell'}({\mathfrak s})\cdot\kappa_{\ell}({\mathfrak s})$ irreducible summands.
\end{Theorem}

We need to slightly generalize Theorem~\ref{TMainSL}. Consider the intermediate subgroups $SL_n\leq H\leq G\leq GL_n$, and let $V$ be an irreducible $\FF G$-module which is an irreducible constituent of $W\dar_G$ for an irreducible $\FF GL_n$-module $W=L({\mathfrak s})$. Then by Lemma~\ref{Llink2}, we have 
\begin{equation}\label{EB1}
\kappa^G_H(V)=\kappa^{GL_n}_H(W)/\kappa^{GL_n}_G(W).
\end{equation}
So to compute $\kappa^G_H(V)$ for two intermediate subgroups $G$ and $H$, it suffices to compute $\kappa^{GL_n}_H(W)$ for any intermediate subgroup $H$. For such $H$ define $H\leq B\leq GL_n$ so that $B/H = \OLB(GL_n/H)$. As $GL_n/B$ is a cyclic $\ell'$-group, Lemmas~\ref{linear} and \ref{tensor} allow us to compute 
$\kappa^{GL_n}_B(W)$ as the number of those $\ell'$-elements $\tau\in\FQ^\times$ which leave the symbol ${\mathfrak s}$ invariant and such that the corresponding  representation $L(\tau,(n))$ factors through $B$. Finally, we use Theorem~\ref{TMainSL} to compute $\kappa^{GL_n}_{SL_n}(W)_\ell$ and then apply Lemma~\ref{link2} 
to evaluate 
\begin{equation}\label{EB2}
\kappa^{GL_n}_H(W)= \kappa^{GL_n}_{B}(W) \cdot \min\{ (\kappa^{GL_n}_{SL_n}(W))_{\ell}, |GL_n/H|_{\ell}\}.
\end{equation}

For future reference we state some results on Clifford theory from \cite{KTSL}.

\begin{Lemma}\label{rami}  {\rm \cite[Lemma~3.3]{KTSL}} 
{ Let $r$ be a prime, $S \lhd G$ with $G/S$ cyclic, 
$V\in\IBr(G)$, and $S \leq A,B \leq G$ be such that $A/S = O_{r}(G/S)$ and  
$B/S = O_{r'}(G/S)$. Also, let $U$ (resp. $W$) be an irreducible constituent of $V\dar_{A}$ 
(resp. $V\dar_{B}$). Then
\begin{enumerate}
\item $\kappa^{G}_{S}(V) = \kappa^{G}_{A}(V) \cdot \kappa^{G}_{B}(V)$;
\smallskip
\item $\kappa^{G}_{A}(V) = \kappa^{B}_{S}(W)$, $\kappa^{G}_{B}(V) = \kappa^{A}_{S}(U)$.
\end{enumerate}}
\end{Lemma}

\begin{Lemma}
\label{2mod} {\rm \cite[Lemma~3.7]{KTSL}} 
Let $S\lhd G$ with $G/S$ cyclic, $S\leq A\leq G$ with $A/S=O_\ell(G/S)$, 
and $U \in \IBr(S)$. Assume that $U$ is an $S$-composition factor of $V_{i}\dar_{S}$ for some
$V_{i} \in \IBr(G)$, $i = 1,2$. Then 
$V_{2} = V_{1} \otimes L$ for some $L \in \IBr(G/A)$.
\end{Lemma}

\begin{Lemma}
\label{1bl} {\rm \cite[Lemma~3.9]{KTSL}} 
 Let $S\lhd G$ with $G/S$ an $\ell$-group, $V$ an $\FF G$-module, and 
\begin{enumerate}
\item[{\rm (i)}] $V\dar_{S}$ has a filtration $0 = V^{0} < V^{1} < \ldots$, where
each $V_{i} := V^{i}/V^{i-1}$ is a $G$-conjugate of 
$V_{1} = V^{1}$;
\item[{\rm (ii)}] All composition factors of $V_{1}$ belong to the same $\FF S$-block.
\end{enumerate}
Then all composition factors of $V$ belong to the same $\FF G$-block.
\end{Lemma}

We will often use the following facts from Clifford theory \cite[IV.4.10]{F}. If $H\lhd G$, $B$ is a $G$-block and $W\in\IBr(G)\cap  B$ then there exists an $H$-block $b$ such that all constituents of $W{\dar}_H$ belong to $\cup_{g\in G}b^g$. In this situation we say that $B$ {\em covers} $b$. Moreover, if $B$ covers $b$ then for any $V\in\IBr(H)\cap b$ there is some $W\in \IBr(G)\cap  B$ such that $V$ appears as a constituent of $W{\dar}_H$. 

\subsection{Some large subgroups}
We list certain large subgroups in $GL_{n}$:

\begin{Proposition}\label{large}
Let $n \geq 5$, $q = p^{f}$, and $SL_n(q)\not\leq H<GL_{n}(q)$. Assume that one of the following two conditions holds:
\begin{enumerate}
\item[{\rm (a)}] $|H/Z(H)| > q^{(n^{2}+5)/2}$;
\item[{\rm (b)}] $2|n$, $|H/Z(H)| > q^{n^{2}/2-4}$, and a p.p.d. $r$ for $(p, (n-1)f) $  divides $|H|$.
\end{enumerate}
Then one of the following holds:
\begin{enumerate}
\item[{\rm (i)}] $H$ is contained in a proper parabolic subgroup of $GL_n(q)$;

\item[{\rm (ii)}] $H \leq \Gamma L_{d}(q^{s})$ with $1 < s = n/d$;
 
\item[{\rm (iii)}] $n$ is even, (b) does not hold, and $Sp_n(q)\leq H \leq CSp_{n}(q)$;

\item[{\rm (iv)}] $q$ is a square, (a) does not hold, and $H \leq CU_{n}(\sqrt{q})$.
\end{enumerate}
\end{Proposition}

\begin{proof}
If $n \neq 6$, then the conditions on $H$ imply that $\bar{H} := H/(Z(GL_n(q)) \cap H)$ is a subgroup 
of order $> q^{3n}$ of $PGL_{n}(q)$ and so we can apply the main result of \cite{Lieb1} 
to $\bar{H}$. If $n = 6$, we can argue directly using Aschbacher's Theorem~\cite{A}. In either case, one of the following  must occur. 

1) $H$ preserves some nonzero proper subspace $W$ of the natural $\FQ GL_n(q)$-module $\NC = \FQ^{n}$, 
i.e. (i) holds. 

2) $H$ preserves some decomposition $\NC = V_{1} \oplus \ldots \oplus V_{k}$ with 
$\dim V_{i} = d = n/k$ and $k > 1$, and so $H \leq GL_{d}(q) \wr \SSS_{k}$. In particular, 
$|H| < q^{n^{2}/2+1}$. Furthermore, if $k < n$ then $r{\not{|}}\,|H|$, and if  
$k = n$, then $|H| \leq (q-1)^{n} \cdot n! < q^{n^{2}/2-4}$. So this case does not happen.

3) $H$ preserves an extension field structure of $\NC$, i.e. (ii) holds.

4) $H$ preserves a subfield structure of $\NC$. Then 
$|H/Z(H)| \leq |GL_{n}(\sqrt{q})| < q^{n^{2}/2}$, and $r {\not{|}}\,|H|$, ruling out this case. 

5) $H$ stabilizes a tensor decomposition $\NC = A \otimes B$, and so 
$H \leq GL_{a}(q) \otimes GL_{b}(q)$ with $a, b < ab = n$. Hence 
$|H| < q^{a^{2}+b^{2}} \leq q^{n^{2}/2}$, and $r {\not{|}}\,|H|$.

6) $H$ permutes the factors of some tensor decomposition $\NC = V_{1} \otimes \ldots \otimes V_{s}$,
$s > 1$ and $d := \dim V_{i} > 1$. Then $|H|$ divides $|GL_{d}(q)|^{s} \cdot s!$
and so it is coprime to $r$ and smaller than $q^{n^{2}/2+1}$.

7) $H$ is contained in the normalizer of the action of $\AAA_{c}$ or $\SSS_{c}$ on its smallest 
module $\NC$ and $n \in \{c-1,c-2\}$. Then $|H/Z(H)| \leq (n+2)! < q^{n^{2}}$. If $2|n$,
then $|H/Z(H)| < q^{n^{2}/2-4}$ unless $(n,q) = (6,2)$, in which case 
$r = 31{\not|}\,|H|$.

8) $n = t^{m}$ for some prime $t {\not{|}}\,q$ and $H$ normalizes a $t$-group $T$ of symplectic type. 
In this case, $|H/Z(H)| \leq t^{2m^{2}+3m}$, so this case does not happen. 

9) $H$ is contained in the normalizer in $GL_n(q)$ of a classical group of dimension $n$ over $\FQ$. 
Notice that $|CO_{n}(q)| < q^{n^{2}/2}$, and if $n$ is even, then 
$r{\not{|}}\,|CO^{\pm}_{n}(q)$. So we arrive at one of the following two cases. 

Case 1: 
$n$ is even and $H \leq CSp_{n}(q)$. Then $r{\not{|}}\,|H|$, and (b) does not hold. Moreover, by (a),    
$H$ is a subgroup of $CSp_{n}(q)$ of index $< q^{(n-5)/2}(q-1)$, whence $H \geq Sp_{n}(q)$ \cite[Table 5.2.A]{KL}. So we arrive at (iii). 

Case 2: $q$ is a square and $H \leq CU_{n}(\sqrt{q})$.     Then 
\begin{eqnarray*}
|CU_{n}(\sqrt{q})/Z(CU_{n}(\sqrt{q}))| &\leq& |PGU_{n}(\sqrt{q})| = q^{n(n-1)/4}\textstyle\prod^{n}_{i=2}(q^{i/2}-(-1)^{i})\\
 & <& 2q^{n(n-1)/4 + n(n+1)/4-1} \leq q^{n^{2}/2},
 \end{eqnarray*}
contradicting (a). We arrive at (iv). 
\end{proof}

\section{Basic reductions}
\subsection{Reduction to subgroups transitive on ${\mathbf 1}$-spaces}\label{SSRed}
The following lemma is key. A similar idea appears in the study of irreducible restrictions from alternating type groups, see \cite{Saxl},  \cite[3.9]{KS1}, \cite[3.4]{BK}, \cite[2.5]{KS2}, \cite[7.4]{KT}. 

\begin{Lemma}\label{hom1}
Let $G$ be a finite group and $V \in \IBr(G)$. Assume that $P < G$ is a subgroup such that the following conditions hold:
\begin{enumerate}
\item[{\rm (i)}] $\dim\End_{P}(V\dar_{P}) \geq 2$; 
\item[{\rm (ii)}] The module $W := (1_{P}) \uar^{G}$ is either a direct sum $1_{G} \oplus A$ or 
a uniserial module $(1_{G}|A|1_{G})$ with composition factors $1_{G}$ and $A\not\cong1_G$.
\end{enumerate}
Then, for any subgroup $H \leq G$, either $G = HP$ or $V\dar_H$ is reducible.
\end{Lemma}

\begin{proof}
Assume that $G \neq HP$. Then the Mackey decomposition 
$$\textstyle W\dar_{H} = \bigoplus_{HgP \in H \backslash G/P}(1_{H \cap P^{g}})\uar^H$$ has more than one summand. 
Set $U := (1_{H}) \uar ^{G}$. It follows that 
\begin{eqnarray*}
 \Hom_{G}(U,W) &\cong &\Hom_{H}(1_{H},W \dar_{H}) 
\\
&\cong &
\textstyle  \bigoplus _{HgP \in H \backslash G/P}\Hom_{H \cap P^{g}}(1_{H \cap P^{g}},1_{H \cap P^{g}})
\end{eqnarray*}
has dimension $\geq 2$. One of the elements in $\Hom_{G}(U,W)$ is characterized uniquely up to scalar by the property that its image is precisely $1_G\subset W$. 
So there must also exist a $G$-homomorphism  
$\psi:U \to W$ 
whose image is not contained in the unique $G$-submodule $1_G$ of $W$. 
Set $B := A$ (resp. $B := (1_{G}|A)$) if $W = 1_{G} \oplus A$ (resp. if $W = (1_{G}|A|1_{G})$). Then $\im\psi\supset B$. 

On the other hand, by (i), we have
\begin{equation*}
2 \leq \dim\Hom_{P}(V,V)
= \dim\Hom_{P}(1_P,\End(V))
= \dim\Hom_{G}(W,\End(V)).
\end{equation*}
One of the elements  $\chi\in\Hom_{G}(W,\End(V))$ is characterized uniquely up to scalar by the property that its image is 
precisely $1_G$. Then $\ker \chi=B$.  
So there also exists a $G$-homomorphism  
$\phi:W\to\End(V)$ with $\Ker(\phi) \subsetneq B$. 

It now follows that $\phi\circ\psi:U\to \End(V)$ is a $G$-homomorphism whose image is not contained in  $1_{G}$, for otherwise $\phi(B) \subseteq 1_{G}$, and so $\phi(B) = 0$, giving a contradiction. 
But there also exists a $G$-homomorphism $\theta:U\to\End(V)$ with $\Im(\theta) = 1_{G}$. So $\theta$ and $\phi\circ\psi$ are two linearly independent elements of $\Hom_{G}(U,\End(V))$. So 
$\dim\Hom_{G}(U,\End(V)) \geq 2$. 
Finally, 
\begin{eqnarray*}
\dim\Hom_{H}(V\dar_H,V\dar_H)
& = &\dim\Hom_{H}(1_H,\End(V) \dar_{H})
\\ 
& = & \dim\Hom_{G}(U,\End(V))
\ \geq\  2.
\end{eqnarray*}
In particular, $V\dar_{H}$ is reducible.   
\end{proof}


Recall that $P_1$ denotes a special parabolic in $GL_n$, see \S\ref{SMoreNot}.

\begin{Corollary}\label{CRed}
Let $n\geq 3$, $SL_{n}(q) \leq G \leq GL_{n}(q)$, and suppose that  
$V \in \IBr(G)$ is irreducible over a subgroup $H \leq G$. Then either $\dim\End_{P_1\cap G}(V) = 1$ or $H$ acts transitively
on the $1$-spaces of $\NC$. 
\end{Corollary}
\begin{proof}
The condition (ii) of Lemma~\ref{hom1} is satisfied for $P=P_1\cap G$, see e.g. \cite{Mortimer}. By the lemma, the irreducibility of $H$ on $V$ implies that $G = H(P_1\cap G)$, which is equivalent to $H$ acting transitively
on the $1$-spaces of $\NC$. 
\end{proof}

\subsection{Hering's Theorem}\label{SSHering}
In order to handle the second case arising in Corollary~\ref{CRed}, we use Hering's theorem on the groups transitive on $1$-spaces.

\begin{Proposition}\label{step1}
Let $n\geq 3$, $Z := Z(GL_{n}(q))$, and let $H \leq GL_n(q)$ act transitively
on the $1$-spaces of $\NC$. Then
one of the following holds:
\begin{enumerate}
\item[{\rm (i)}] $H \rhd  SL_{a}(q_{1})$ with $q_{1}^{a} = q^{n}$ and $a \geq 2$;

\item[{\rm (ii)}] $H \rhd  Sp_{2a}(q_{1})'$ with $q_{1}^{2a} = q^{n}$ and $a \geq 2$;

\item[{\rm (iii)}] $H \rhd  G_{2}(q_{1})'$ with $q_{1}^{6} = q^{n}$ and $2|q$;

\item[{\rm (iv)}] $HZ$ is contained in $\Gamma L_{1}(q^{n})$;

\item[{\rm (v)}] $(q^{n},HZ)$ is $(3^{4}, \leq 2^{1+4}_{-}\cdot \SSS_{5})$, $(3^{4}, \rhd SL_{2}(5))$, 
$(2^{4},\AAA_{7})$ or $(3^{6},SL_{2}(13))$.
\end{enumerate}
Moreover, in the cases {\rm (i)-(iii)}, $q_{1} = q^{s}$ is a power of $q$, and 
$H \leq \Gamma L_{n/s}(q^{s})$.
\end{Proposition}

\begin{proof}
By assumption, $HZ$ acts transitively on the nonzero vectors of $\NC$. So 
$\NC\rtimes HZ$ acts doubly transitively on $\NC$, with $\NC$ acting via translations and $HZ$ being the 
stabilizer of the zero vector. In this situation, we can apply Hering's theorem as given in \cite{Lieb2} 
to $HZ$ to arrive to (i)-(v). Note that if $HZ$ contains a perfect normal subgroup $R$, then 
$R = [R,R] \leq [HZ,HZ] = [H,H]$ and so $R \lhd H$. This is why in the cases (i)-(iii) we do not have $Z$. 

It remains to show that in the cases (i)-(iii), $q_{1} = q^{s}$ is a power of $q$, and 
$H \leq \Gamma L_{n/s}(q^{s})$. This is clear in
the case $H \rhd Sp_{4}(2)'$, so let $H \rhd S$, where $S = SL_{a}(q_{1})$, $Sp_{2a}(q_{1})'$ with $(a,q_{1}) \neq (2,2)$, or $G_{2}(q_{1})'$. 
Then the smallest 
degree $d$ of nontrivial irreducible projective representations of $S/Z(S)$ over fields of characteristic
$p$ dividing $q$ is $d = a$, $d = 2a$, or $d = 6$, respectively, cf. \cite[
5.4.13]{KL}. So $q_{1}^{d} = q^{n}$. 

Since $H$ is transitive on the $1$-spaces of $\NC$, $H$ is irreducible on $\NC$, 
and so $\NC\dar_{S} = W_{1} \oplus W_{2} \oplus \ldots \oplus W_{t}$ is a direct sum of 
irreducible $\FQ S$-modules $W_{i}$ (of the same dimension $n_{1} := n/t$). Then $W_{1}$ can be viewed as an 
absolutely irreducible $S$-module $W_{0}$, over $\End_{S}(W_{1}) \cong \FF_{q^{s}}$ for some integer 
$s \geq 1$. 

Write $q_{1} = p^{a}$, and assume that the smallest field over which the 
absolutely irreducible module $W_{0}$ can be realized is $\FF_{p^{b}}$. Clearly, $q^{s} \geq p^{b}$.
By 
\cite[5.4.4, 5.4.6]{KL}, $b|a$ and there is an irreducible $\bar{\FF}_{p}S$-module 
$M$ such that $W_{0} \otimes \bar{\FF}_{p}$ is isomorphic to the tensor product of $a/b$ Frobenius 
twists of $M$; in particular, $\dim W_{0} = (\dim M)^{a/b}$. 
Note that $\dim M \geq d$. Otherwise 
$\dim M = 1$, $\dim W_{0} = 1$, $S$ acts trivially on each $W_{i}$ and on $\NC$, a contradiction as
$H$ acts faithfully on $\NC$. It follows that
$$p^{da} = q_{1}^{d} = q^{n} \geq q^{n_{1}} = |W_{1}| = |W_{0}| = q^{s\dim W_{0}} 
  \geq p^{b\dim W_{0}} \geq p^{b d^{a/b}}.$$  
On the other hand, $d^{a/b-1} \geq a/b$ since $d \geq 2$ and $b|a$, with equality exactly when 
either $a = b$, or $a/b = 2 = d$. Thus $b \cdot d^{a/b} \geq da$.
Consequently, $n = n_{1}$, $q^{s} = p^{b}$, $\dim W_{0} = d^{a/b} = da/b$,
$q_{1}^{d} = q^{sda/b}$, and $q_{1} = q^{sa/b}$. 

If $a = b$, then $q_1 = q^{s}$, $d = n/s$, and $\NC\dar_{S} = W_{1}$ can be viewed as the $d$-dimensional 
module $W_{0}$ of $S$. Using \cite[
5.4.11]{KL}, one can show that, up to a twist by an automorphism 
of $S$, $W_{0}$ is isomorphic to the natural module $\FF_{q^{s}}^{d}$ of $S$, whence
$H \leq \Gamma L_{d}(q^{s})$.  

Finally, 
let $a/b = 2 = d$, and so $S = SL_2(q^{2s})$. Here $W_{0} \otimes \bar{\FF}_{p}$ is the 
tensor product of $M = \bar{\FF}_{p}^{2}$ and the $q^{s}$-Frobenius twist of $M$, so it 
affords a 
unique up to scalar $S$-invariant quadratic form $Q$, cf. \cite[p.45]{KL}. As $H \rhd S$, $H$ preserves $Q$ up 
to a scalar. Hence $H$ preserves the set of nonzero $Q$-singular $1$-spaces in  $W_{0}$, and so 
it cannot act transitively on the $1$-spaces of $\NC$. 
\end{proof} 

\subsection{Further reductions}
In this subsection we exclude the semilinear subgroups $H \leq \Gamma L_{d}(q^{s})$ with $s = n/d > 1$. The case 
$s = n$ is quite easy:  

\begin{Lemma}\label{semi1}
Let $n\geq 3$, $SL_{n}(q) \leq G \leq GL_{n}(q)$, $V \in \IBr(G)$, $\dim V>1$, and $H \leq G$ be a subgroup of
$\Gamma L_{1}(q^{n})$. Then $V$ is irreducible over $H$ if and only if $\ell\neq 7$ and $(G,H,\dim V) = (SL_{3}(2),C_{7}:C_{3},3)$.
\end{Lemma}

\begin{proof}
We have $K := \Gamma L_{1}(q^{n})\cong C_{q^{n}-1}:C_n$. Hence by \cite[
(6.15)]{Is}, $\mc(K)\leq n$. So $\dim V\leq \ml(H) \leq \mc(H) \leq \mc(K) \leq n$. 
If 
$(n,q) \neq (3,2)$, $(3,4)$, $(4,2)$, $(4,3)$, then by \cite{GT1}, 
$\dim V \geq \dl(SL_{n}(q)) \geq (q^{n}-q)/(q-1)-1 > n$, giving a contradiction. If $(n,q) = (4,3)$, 
$(4,2)$, or $(3,4)$, then one can check that $\dim V \geq 7 > n$, giving a contradiction again. Finally,
if $(n,q) = (3,2)$, we get the exceptional case. 
\end{proof}

The case $1 < s < n$ is non-trivial. Note that the semilinear groups are not
considered in \cite{S2} as their actions on $G$-modules are imprimitive. The exception in Proposition~\ref{semi2} will be ruled out in Proposition~\ref{semi3}.

\begin{Proposition}\label{semi2}
Let $n\geq 3$, $SL_{n}(q) \leq G \leq GL_{n}(q)$, and $H \leq G$ is a subgroup of
$\Gamma L_{d}(q^{s})$ with $1 < s = n/d < n$. If $V \in \IBr(G)$ is irreducible over $H$ and 
$\dim V  > 1$, then $\ell \neq 2$, $G = SL_{n}(3)$, $s = 2$, and $d = n/2$ is odd.
\end{Proposition}

\begin{proof}
We may assume that $H = G \cap \Gamma L_{d}(q^{s})$. We may also assume that 
$s$ is prime, as 
$\Gamma L_{n/st}(q^{st}) \leq \Gamma L_{n/s}(q^{s})$ whenever $st|n$. We consider two cases:

Case 1: there is a p.p.d. $r$ for $(q,s)$. 
Set $R := O_{r}(Z(GL_{d}(q^{s})))$. Note that 
$R = \langle g \rangle$ is cyclic and $R \lhd H$ because $r{\not{|}}(q-1)$. 
In fact, $H/C_{H}(R)$ is a cyclic group of order dividing $s$, as 
$H \leq GL_{d}(q^{s}) \cdot C_{s}$ and $GL_{d}(q^{s})$ centralizes $R$. 

Note further that $\ell \neq r$, as otherwise $0 \neq C_{V}(g) \lneq V$ is fixed by $H$. 
Also, $g$ does not act trivially on $V$ because $r{\not{|}}(q-1)$.
Let $\dvg$ be the degree of the minimal polynomial of $g$ acting on $V$. So $g$ has exactly $\dvg$ eigenvalues on $V$, and the eigenspaces are permuted by 
$H/C_{H}(R)$. Now the irreducibility of $H$ on $V$ implies $\dvg \leq s$.

Note that 
$r^{a}:=|g|$ 
equals the order of $gZ(G)$ in $G/Z(G)$ as $r$ is coprime to the order of 
$Z(G) \leq C_{q-1}$. By the main result of \cite{TZ}, one of the following holds:
(i) $\dvg > r^{a-1}(r-1)$;
(ii) $r > 2$, $\dvg = r^{a-1}(r-1)$, and Sylow $r$-subgroups of $G/Z(G)$ are cyclic. 
This leads to a contradiction, as in (i) or if $a \geq 2$ in (ii), we have $\dvg \geq r > s$, and in 
(ii) with $a = 1$ we have $s \geq \dvg = r-1 \geq s$ implying $r = s+1$. 
As $r$ and $s$ are primes, we get $r = 3$, $s = 2$, $n = ds \geq 4$, and $p \neq 3$, 
whence the Sylow $r$-subgroups of $G/Z(G)$ are not cyclic.


Case 2: there is no p.p.d. for $(q,s)$. As $s$ is prime, Lemma~\ref{LPPD} implies that $s = 2$ and $q = p = 2^{t}-1$ is a Mersenne prime. Set 
$r = 2$ and consider $R := O_{r}(Z(GL_{d}(q^{s}))) = \langle g \rangle$ as above. Then $|g| = 2^{t+1}$. 
Set $H \cap R = \langle h \rangle$. Note that $H$ is a normal subgroup of $\Gamma L_d(q^s)$ with $|\Gamma L_d(q^s):H|_2\leq 2$. So $|h| = 2^{t_{1}+1}$ and $hZ(G)$ has order $2^{t_{1}}$ in $G/Z(G)$ 
with $t-1 \leq t_{1} \leq t$. Moreover, $|H/C_{H}(h)| \leq 2$. As 
above, the irreducibility of $H$ on $V$ now implies that $\ell\neq 2$ and $d_{V}(h) \leq 2$.  As the Sylow $2$-subgroups 
of $G/Z(G)$ are not cyclic, we have $d_{V}(h) > 2^{t_{1}-1}$ by \cite{TZ}. So 
$t = 2$, $t_1=1$, $q = 3$. Since $t_1=t-1$, we see that $R\not\leq H$. In particular $G\neq GL_n(q)$, i.e. $G=SL_n(3)$. Finally, if $d$ is even, then $g\in SL_n(q)$, so $R\leq H$, giving a contradiction.
\end{proof}

\section{Quantum $GL_N$}\label{SQG}
\subsection{Overview}\label{SSO}
We want to exploit relations between quantum linear groups and representation theory of $GL_n(q)$ over the field $\FF$ of non-defining characteristic \cite{BDK}. We start with necessary definitions and some known results. 
 
Let $N\in {\mathbb N}$, $t$ be an indeterminate, $R={\mathbb Z}[t,t^{-1}]$, and
$K={\mathbb Q}(t)$. The {\em quantized
enveloping
algebra of $GL_N$} (or quantum $GL_N$) is the $K$-algebra $U(N)_K$  with generators
$
E_i,\ F_i,\ K_j,\ K_j^{-1},\ 1\leq i<N,\ 1\leq j\leq N,
$
and certain well-known relations, see for example \cite[$\S$1]{BK}.
For $m, n\in\ZZ_{\geq 0}$ set
$$
\textstyle[m]:=\frac{t^m-t^{-m}}{t-t^{-1}}, \quad [m]!:=\prod_{i=1}^m[i], \quad 
{m\choose n}
:=\frac{[m][m-1]\dots[m-n+1]}{[n]!}.
$$
Denote by $U(N)_R$ the $R$-subalgebra of $U(N)_K$ generated by 
$$
\textstyle E_i^{(s)},\ F_i^{(s)},\ K_j,\ K_j^{-1},\
{K_j\choose m}
\quad
(s,m\in{\mathbb Z}_{\geq 0},\ 1\leq i<N, \ 1\leq j\leq N),
$$
where
$$
\textstyle 
{K_j\choose m}
:=\prod_{i=1}^m\frac{K_jt^{-m+1}-K_
j^{-1}t^{m-1}}{t^i-t^{-i}}\quad \text{and}\quad
X^{(s)}:=\frac{X^s}{[s]!}\
\text{for $X\in U(N)_K$}.
$$

Let $v\in\FF^\times$. We regard $\FF$ as an $R$-module by
letting $t$ act by multiplication with $v$. Set
$
U(N)=U_v(N):=U(N)_R\otimes_R \FF,
$ 
and let us use the (same) notation $X^{(s)},
{K_j\choose m}
$, etc. for $X^{(s)}\otimes
1,
{K_j\choose m}
\otimes 1$, etc.
We always consider $U(N)$ as a subalgebra of $U(N+1)$ generated by
$$
\textstyle \{E_i^{(s)},\ F_i^{(s)},\ K_j,\ K_j^{-1},\
{K_j\choose m}
\mid s,m\in{\mathbb Z}_{\geq
0},\ 1\leq
i<N, \ 1\leq j\leq N\}.
$$

Let $U(N)^+$ be the subalgebra of $U(N)$ generated by all $E_i^{(s)}$, and $U(N)^0$ be the subalgebra of $U(N)$ generated by all $K_j,\
K_j^{-1},\
{K_j\choose m}
$. Set $\La(N)=({\mathbb Z}_{\geq 0})^N$, and 
\begin{eqnarray*}
\La^+(N)&=&\{\la=(\la_1,\dots,\la_N)\in\La(N)\mid \la_1\geq\dots\geq
\la_N\},\\
\La(N,r)&=&\{\la=(\la_1,\dots,\la_N)\in\La(N)\mid
\la_1+\dots+\la_N=r\},\\
\La^+(N,r)&=&\La^+(N)\cap \La(N,r).
\end{eqnarray*}
We identify $\la\in\La(N)$ with a {\em weight}, i.e. the $\FF$-algebra
homomorphism
$$
\textstyle 
\la:U(N)^0\rightarrow \FF,\ K_j\mapsto v^{\la_j},\
{K_j\choose m}
\mapsto
{\la_j\choose m}.
$$

For every $\la\in \La^+(N)$ there exists a unique
irreducible $U(N)$-module $L_N(\la)$ with highest weight $\la$. Moreover, if $\la\in\La^+(N,r)$, then 
$$
\textstyle L_N(\la)=\bigoplus_{\mu\in \La(N,r)} L(\la)_\mu,
$$
where for a $U(N)$-module $M$, its {\em $\mu$-weight space} $M_\mu$ is
defined as follows:
$$
M_\mu:=\{v\in M\mid Kv=\mu(K)v\ \text{for any $K\in U(N)^0$}\}.
$$


Recall that $q$ is a prime power with $(\ell,q)=1$. Let $d$ be a fixed positive
integer. Set
$s=s(d)$ to be the minimal positive integer such
that $$q^{ds}\equiv 1\pmod{\ell}.$$
In other words $q^d$ (considered as an element $q^d\cdot 1_{\FF}\in\FF$) is the primitive $s$th root of unity in $\FF$. By convention, $s=\infty$ if $\ell=0$.
Let $v$ be a square root of $q^d$ in $\FF$ such that if $s$ is odd  then
$v$ is also a primitive $s$th root of unity (such $v$ always exists,
see
\cite[\S1.3]{BDK}). {\em From now on, we assume that it is always this $v$ that is
used in the definition of the quantum group $U(N)=U_v(N)$}.

It is also convenient to define $e=e(d)$ to be the smallest positive
integer  such that in $\FF$ we have 
$
\sum_{i=0}^{e-1}q^{di}=0. 
$
For example, assuming $\ell$ is positive, $e=s$ if $q^d\neq 1$ in $\FF$, and $e=\ell$
if $q^d=1$ in $\FF$. If $\ell=0$ then $s=\ell=e=\infty$.

For a positive integer $r$, a partition
$\la=(\la_1,\dots,\la_N)\in\La^+(N)$ is called {\em $r$-restricted} if
$\la_i-\la_{i+1}<r$ for all $i=1,2,\dots,N-1$. By convention, $\la$ is
$1$-restricted if and only $\la=(0)$.
By an {\em $(s,\ell)$-adic expansion} of $\la\in\La^+(N)$ we mean some
(non-unique) way of writing
\begin{equation}\label{E486}
\la=\la^{-1}+s\la^0+s \ell\la^1+s \ell^2\la^2+\dots
\end{equation}
such that $\la^{-1}\in\La^+(N)$ is $s$-restricted and 
$\la^i\in\La^+(N)$ is $\ell$-restricted for each $i\geq 0$. 

For each non-negative integer $r$ there is the {\em
$r$th Frobenius twist} operation $M\mapsto M^{[r]}$, which associates a
$U(N)$-module $M^{[r]}$ to a polynomial $GL_N(\FF)$-module $M$ so that the
weights of $M^{[r]}$ are obtained from the usual $GL_N(\FF)$-weights of $M$
by multiplication with $s \ell^r$, see \cite[$\S$1.3]{BDK} for 
details. Now, the Steinberg Tensor Product Theorem in this context is (see e.g. \cite[1.3e]{BDK}):

\begin{Theorem}\label{TSTP}
If $\la\in \La^+(N)$ has an $(s,\ell)$-adic decomposition
{\rm (\ref{E486})} then
$$
L_N(\la)\cong L_N(\la^{-1})\otimes L_N(\la^0)^{[0]}\otimes
L(\la^1)^{[1]}\otimes\dots
$$
\end{Theorem}

Steinberg Tensor Product Theorem~allows
us to determine weight system of any irreducible $U(N)$-module
$L_{N}(\la)$ if we take into account the following 

\begin{Theorem}\cite[
1.4]{BK}\label{TPS}
Let $\la\in \La^+(N)$ be $e$-restricted. Then $L_N(\la)_\mu\neq
0$ if and only if
$\mu$ is a weight of the irreducible polynomial $GL_N({\mathbb C})$-module
with highest weight $\la$, i.e. if and only if $\mu^+\leq \la$, where the
partition $\mu^+\in\La^+(N)$ is obtained from the composition $\mu$ by
permuting parts, and $\leq$ is the dominance order on partitions.
\end{Theorem}

 For future reference we introduce
 
\begin{Definition}\label{DEDiv}
{\rm 
A non-zero partition $\la$ is called {\it $e$-divisible} if $e$ divides all parts of $\la'$. 
If $e=\infty$, the zero partition is the only $e$-divisible partition by convention.
}
\end{Definition}

So $e$-divisible partitions are those for which  $L_N(\la')$ is a Frobenius twist.


\subsection{Some results on branching for quantum $\mathbf{GL_N}$}
The following notion is convenient when studying restrictions from $U(N)$ to $U(N-1)$.  
\begin{Definition}
\label{DLev}
{\rm
Let $\la\in \La^+(N)$.
For $j\geq 0$ define the {\em $j$-th level} of the irreducible
$U(N)$-module $L_{N}(\la)$ to be the subspace
$$
L_{N}(\la)[j]:=\bigoplus_{\mu\ {\rm such\ that}\ \mu_{N}-\la_{N}=j}L_{N}(\la)_\mu.
$$
}
\end{Definition}

The following is easy to check:

\begin{Lemma}
\label{LDirSum}
Let $\la\in\La^+(N,k)$. Each
level $L_{N}(\la)[j]$ is a submodule of
$L_{N}(\la)\dar_{U(N-1)}$, we have 
$
L_{N}(\la)\dar_{U(N-1)}=\oplus_{j\geq 0}L_{N}(\la)[j],
$
and all $U(N-1)$-composition factors of the $j$th
level $L_{N}(\la)[j]$ are of the form $L_{N-1}(\mu)$ for $\mu\in
\La^+(N-1,k-j)$.
\end{Lemma}


\begin{Lemma}
\label{LLevNonTr}
If $L_{N}(\la)[j] \neq (0)$ then $0\leq j\leq \la_1-\la_{N}$.
Moreover, if $\la$ is $e$-restricted, then the converse is true.
\end{Lemma}

\begin{proof}
This follows from Theorem~\ref{TPS} and the corresponding fact for
polynomial representations of $GL_{N}({\mathbb C})$, which is an easy
exercise.
\end{proof}

\begin{Lemma}\label{LLevelNew}
Let $N> k$ and $\la\in\La^+(N,k)$. Then $L_{N}(\la)[1]=(0)$ if and
only if $e$ divides all parts of $\la$, i.e. $\la'$ is $e$-divisible
\end{Lemma}

\begin{proof}
Follows from Theorem~\ref{TSTP}, 
weight considerations and Lemma~\ref{LLevNonTr}.
\end{proof}

\begin{Proposition}
\label{POneIrr}
Let $N> k$ and $\la\in\La^+(N,k)$. Assume that
there is $j> 0$ such that $L_{N}(\la)[j]\neq 0$ and $L_{N}(\la)[i]= 0$
for all positive integers $i$ different from $j$.
Then $\la$ is of the form $(a^m)$ with $a=1$ or $a=s\ell^r$ for some $r\in\ZZ_{\geq 0}$ and $m\in\NN$, in which case $j=a$ and $L_{N}(\la)[a]\cong L_N((a^{m-1}))$.
\end{Proposition}

\begin{proof}By Theorem~\ref{TSTP}, we may assume that $\la$ is $e$-restricted. Then by Lemma~\ref{LLevNonTr}, every level $L[j]$ with $0\leq j\leq
\la_1-\la_{N}$
is non-trivial. So we have $\la_1-\la_{N}=1$. As $N> k$, we have
$\la_{N}=0$, hence $\la=(1^k)$. Finally, $(1^k)$ is a minuscule weight,
so
$L((1^k))$ coincides with the Weyl module $\De((1^k))$ of highest weight
$(1^k)$, and the second part of our proposition follows by weight
considerations (cf.
\cite[3.19(i)]{BrBr}).
\end{proof}

\subsection{Affine $\mathbf{GL_n}$ and branching}\label{SAGL}
Recall that we write $GL_n=GL_n(q)$. By convention, $GL_0=\{1\}$.
For
every $0\leq i\leq n$, there is a functor
$$
e_i:\mod(\FF GL_{n-i})\rightarrow \mod(\FF AGL_n),
$$
that sends irreducibles to irreducibles, and any irreducible $\FF AGL_n$-module has a form $e_i L$, where $0\leq i\leq n$
and $L\in\IBr(GL_{n-i})$. Moreover, $e_i L \cong e_j N$ if and only if
$i=j$ and
$L\cong N$, see \cite[$\S$5.1]{BDK} for more details. If $i\neq j$, the
$\FF AGL_n$-modules $e_iL$ and $e_jN$ belong to different blocks \cite[5.1a(ii)]{BDK}.
The following key theorem provides a connection to quantum groups. In the theorem, $U(N)$ is defined using (a square root $v$ of) $q^d$, as explained in \S\ref{SSO}.

\begin{Theorem}
\label{TBr}
\cite[5.4d(ii)]{BDK}
Let $n=dk$, $\sigma$ be an element of degree $d$ over ${\mathbb F}_q$,
$\lambda\vdash k$. 
Then all composition factors of
$L(\si,\la)\dar_{AGL_{n-1}}$ are of the form $e_{dj-1}L(\si,\mu)$
for
$1\leq j\leq k$ and $\mu\vdash (k-j)$. Moreover, if $N>k$, for such $\mu$ and $j$ we have 
$$
[L(\si,\la)\dar_{AGL_{n-1}}:e_{dj-1}L(\si,\mu)]=[L_{N}(\la')\dar
_{U(N)}:L_{N-1}(\mu')].
$$
\end{Theorem}

\begin{Corollary}
\label{CDecomp}
Let $n=dk$, $\sigma$ be an element of degree $d$ over ${\mathbb F}_q$,
$\lambda\vdash k$, and $V=L(\si,\la)$. Then one of the following happens:
\begin{enumerate}
\item[{\rm (i)}] $V\dar_{AGL_{n-1}}$ has composition factors $e_{di-1}L(\si,\nu)$ and $e_{dj-1}L(\si,\mu)$ with $i\neq j$; in particular, the restriction $V\dar_{AGL_{n-1}}$ has composition factors in at least two different blocks. 
\item[{\rm (ii)}] $\la=(m^a)$ with 
$a=1$ or $a=s\ell^r$ for some $r\in\ZZ_{\geq 0}$ and $m\in\NN$
 and $L(\sigma,(m^a))\dar_{AGL_{n-1}}\cong e_{da-1}L(\sigma,(m-1)^a)$. 
\end{enumerate}
\end{Corollary}

\begin{proof}
Recall that the $AGL_{n-1}$-modules $e_iM$ and $e_j N$ lie in different blocks 
for $i\neq j$. So, the result follows from  Theorem~\ref{TBr} and Proposition~\ref{POneIrr}.
\end{proof}

Here is a slightly different version.

\begin{Corollary}\label{RIdent}
Let $n=dk$, $\sigma$ be an $\ell'$-element of degree $d$ over ${\mathbb F}_q$,
$\lambda\vdash k$, and $V=L(\si,\la)$. Then one of the following happens:
\begin{enumerate}
\item[{\rm (i)}] $V\dar_{AGL_{n-1}}$ has composition factors $e_{di-1}L(\si,\nu)$ and $e_{dj-1}L(\si,\mu)$ with $i\neq j$; in particular, the restriction $V\dar_{AGL_{n-1}}$ has composition factors in at least two different blocks. 
\item[{\rm (ii)}] $V\cong  L(\tau,(m))$
for $\tau$ of  degree $da$ with 
$a=1$ or $a=s\ell^r$ for some $r\in\ZZ_{\geq 0}$, $[\tau_{\ell'}]=[\si]$, 
and $L(\tau,(m))\dar_{AGL_{n-1}}\cong e_{da-1}L(\tau,(m-1))$. 
\end{enumerate}
\end{Corollary}
\begin{proof}
Apply Corollary~\ref{CDecomp} and \cite[4.3e]{BDK}.
\end{proof}


\section{Maximal parabolics and dimension bounds}
Throughout the section: $SL_n\leq G\leq GL_n$, 
$P_{n-1}<GL_n$ denotes the maximal parabolic subgroup consisting of the matrices in $GL_n$ with the last row of the form
$(0,\dots,0,*)$, 
and $U$ 
is the unipotent radical of $P_{n-1}$. 

\subsection{Restricting to a maximal parabolic}\label{SSRMP}
In view of Corollary~\ref{CRed}, we need to study irreducible $\FF G$-modules $V$ with
$\dim \End_{P_{1}\cap G}(V)=1$. It is more convenient to work with $P_{n-1}$ instead of $P_1$, but see Remark~\ref{p-hp} below. The following result follows easily from Mackey theorem (cf. e.g. \cite[5.2a]{BDK}):

\begin{Lemma}
\label{LHC}
Let 
$V=V_1\circ\dots\circ V_a\in\IBr(GL_n)$ for $a>1$. 
Then 
$V\dar_{P_{n-1}}$ is decomposable. In
particular, $\dim\End_{P_{n-1}}(V)>1.$
\end{Lemma}

\begin{Theorem}
\label{TP1}
If $V\in\IBr(GL_n)$ then 
$\dim\End_{P_{n-1}}(V)=1$ if and only if $n=dk$ for some positive
integers $d,k$ and $V\cong L(\sigma,(k))$ 
for an element 
$\sigma$ of degree $d$. 
In this
exceptional case, $V\dar_{P_{n-1}}$ is irreducible, and
$L(\sigma,(k))\dar_{AGL_{n-1}}\cong e_{d-1} L(\sigma,(k-1))$.
\end{Theorem}

\begin{proof}
In view of Lemma~\ref{LHC} we may assume that $V$ is Harish-Chandra indecomposable, i.e. $V$ is of the form $L(\si,\la)$. Since $P_{n-1}\cong AGL_{n-1}\times Z(GL_n)$, the theorem follows from the similar result for $AGL_{n-1}$ in place of $P_{n-1}$, which comes from Corollaries~\ref{CDecomp} and~\ref{RIdent}. 
\end{proof}

There might be Harish-Chandra indecomposable constituents of $W\dar_G$ even if $W$ is Harish-Chandra decomposable. At least we have:

\begin{Theorem}
\label{hom22}
Let $SL_{n} \leq G \leq GL_n$, $P=P_{n-1}$, $W\in\IBr(GL_n)$ be Harish-Chandra indecomposable, $V$ be an irreducible constituent of $W\dar_{G}$, and $\dim\End_{P \cap G}(V\dar_{P \cap G}) = 1$. Then 
$W$ is of the form $L(\tau,(k))$.
\end{Theorem}

\begin{proof}
We can write $W=L(\si,\la)$ for some $\ell'$-element $\si$ of degree $d$. Assume $W$ cannot be rewritten as $L(\tau,(k))$ for some (possibly $\ell$-singular) $\tau$. Set $H := AGL_{n-1}$. By Corollary~\ref{RIdent}, $W\dar_{H}$ has composition factors $M := e_{di-1}L(\sigma,\mu)$ and 
$N := e_{dj-1}L(\sigma,\nu)$, with $i \neq j$, in two different $H$-blocks.
We want to prove that $V\dar_{P \cap G}$ also has composition factors in two different block, which of course implies $\dim\End_{P \cap G}(V\dar_{P \cap G}) \geq 2$. 


Suppose for a contradiction that all composition factors of $V\dar_{P \cap G}$ belong to one $(P \cap G)$-block
$b$. Note that $P = Z(GL_n) \times H$ and $Z(GL_n)$ acts on $W$ by scalars. Hence 
$W\dar_{P}$ has composition factors of the form $\hat{M} := X \otimes M$ and $\hat{N} := X \otimes N$ 
for 
$X\in \IBr(Z(GL_n))$, and these are in different $P$-blocks. 
Let $B$ be the $P$-block containing $\hat{M}$. All 
$P \cap G$-composition factors of $W{\dar}_{P \cap G}$ belong to $\cup_{z \in P}b^{z}$, as $PG=GL_n$. So   
$B$ covers $b$ and any $P$-conjugate $b^{z}$. There is an irreducible constituent
$L$ of $\hat{M}\dar_{P \cap G}$ which belongs to $b$. 
Let $B'$ be the $P$-block containing $\hat{N}$. 
Then $B'$ also covers $b$. So there is an irreducible $\FF P$-module $\hat{N}'$ from the 
block $B'$ such that $L$ is an irreducible constituent of $\hat{N}'\dar_{P \cap G}$. 
By Lemma~\ref{2mod},   
$\hat{N}' = \hat{M} \otimes A$ for some $A \in \IBr(P/(P \cap G))$.

We can identify $\IBr(P/(P \cap G))$ with 
$\IBr(GL_n/G)$. So we may write $A = L(\al,(n))$ for some $\ell'$-element $\al$ of degree $1$. Also,
it is straightforward to check that the $H$-modules
$e_{di-1}L(\sigma,\mu) \otimes (L(\al,(n)) \dar_{H})$ and $e_{di-1}L(\al\sigma,\mu)$ are isomorphic. 
Thus $\hat{N}'\dar_{H} \cong e_{di-1}L(\al\sigma,\mu) =: N'$, and 
$\hat{N}' = Y \otimes N'$ for some $Y \in \IBr(Z(GL_n))$. 
Now recall that 
$N = e_{dj-1}L(\sigma,\nu)$ with $i \neq j$. It follows that $N$ and $N'$ belong to different 
$H$-blocks, whence $\hat{N} = X \otimes N$ and $\hat{N}' = Y \otimes N'$ belong to different 
$P$-blocks, a contradiction.        
\end{proof}

\subsection{$\mathbf{U}$-invariants}\label{SUInv} We now study the invariants of the unipotent radical $U$ of $P_{n-1}$. Recall Definition~\ref{DEDiv}. 

\begin{Theorem}
\label{TQ}
Let $n=dk$ for some positive integers $d,k$ and $L= L(\sigma,\la)$,
where $\sigma$ is an element of degree $d$ over ${\mathbb F}_q$, $e=e(d)$, and $\la$ is
a partition of $k$. Then $L^U\neq(0)$ if and only if $d=1$ and $\la$ is not $e$-divisible. 
\end{Theorem}

\begin{proof}
Let $e_iM\in\IBr(AGL_{n-1})$, see \S\ref{SAGL}. Then $(e_iM)^U\neq (0)$
if and only if $i=0$. Now the result follows from Theorem~\ref{TBr} and
Lemma~\ref{LLevelNew}.
\end{proof}



We develop this for a more general situation:

\begin{Theorem}
\label{hom2}
Let $SL_{n} \leq G \leq GL_n$, $W = L(\si_{1},\la^{(1)}) \circ 
\ldots \circ L(\si_{a},\la^{(a)})$, $d_i=\deg(\si_i)$, and 
 $V$ be an irreducible constituent of $W \dar_{G}$.
If $V^{U} = 0$, then for each $i$, either $d_i > 1$ or $\la^{(i)}$ is 
$e(d_i)$-divisible.
\end{Theorem}

\begin{proof}
Set $L_i:=L(\si_{i},\la^{(i)}), i=1,\dots a$. If $d_i = 1$ and $\la^{(i)}$ is not $e(d_i)$-divisible for some $i$, we may assume that $i=a$ since Harish-Chandra induction is commutative. Then $L_a^{U'} \neq 0$ by Theorem~\ref{TQ},
where $U'$ is the unipotent radical of the standard maximal parabolic in $GL_{|\la^{(a)}|}$. As $W$ is the 
induction from a parabolic subgroup $Q$ to $GL_n$ of the $Q$-module 
$L := L_1 \otimes \ldots \otimes L_a$ with trivial action of the unipotent radical $U''$, it follows that $L^{U'U''} \neq 0$. But $U\leq U'U''$ and $L\subset W\dar_{Q}$, 
hence $W^{U} \neq 0$. 

Now $P_{n-1}G = GL_n$ and $GL_n/G$ is cyclic, so 
$W\dar_{G} = V_{1} \oplus \ldots \oplus V_{r}$, where $V_{j} = g_{j}V \in \IBr(G)$ for some 
$g_{j} \in P_{n-1}$ with $g_{1} = 1$.
As $U < G$, we must have $V_{j}^{U} \neq 0$ for some 
$j$.
 But $V_{j} = g_{j}V$ and $g_{j}$ normalizes $U$. So  
$V^{U} \neq 0$.
\end{proof}

\begin{Remark}\label{p-hp}
{\rm 
Results obtained for the parabolic $P_{n-1}$ and an intermediate subgroup $SL_n\leq G\leq GL_n$ have analogues for $P_{1}$ since the inverse-transpose automorphism $\tau$ 
sends $P_{1}$ to a conjugate of $P_{n-1}$, $\tau(G)=G$, and furthermore,  
$L(\sigma,\la)^{\tau} \cong L(\sigma^{-1},\la)$, cf. \cite[(7.33)]{J}.}
\end{Remark}


\subsection{$U$-invariants and dimension bounds} As above, $U=O_p(P_{n-1})$. 

\begin{Lemma}\label{dim1}
Let $L = L(\sigma,\la)$, 
$\deg(\sigma) = d$, $e=e(d)$, and $\la \vdash k = n/d$. Set $N(m,t) := \prod^{m}_{i=1}(t^{i}-1).$ 
\begin{enumerate}
\item[{\rm (i)}] If $\la = (k)$ then $L$ lifts to a $\CC GL_n$-module and we have $\dim L=N(kd,q)/N(k,q^{d})$. 

\item[{\rm (ii)}] $\dim L$ is divisible by $N(kd,q)/N(k,q^{d})$.

\item[{\rm (iii)}] $\dim L > q^{n^{2}/4-1}(q-1)/2$ 
if $d \geq 2$, and
$\dim L > q^{n^{2}/3-1}(q-1)/2$ 
if $d \geq 3$.

\item[{\rm (iv)}] Assume that $d = 1$ and $\la$ is $e$-divisible. Then 
$\dim L  > q^{n^{2}/4-1}(q-1)/2$ if $e \geq 2$, and $\dim L  > q^{n^{2}/3-1}(q-1)/2$ if $e \geq 3$.
\end{enumerate}
\end{Lemma}

\begin{proof}
(i) and (ii) are well-known, see e.g. \cite{DJ2} or \cite[3.5b,5.5d]{BDK}.

(iii) follows from (ii) and the estimate $\prod^{m}_{i=2}(1-q^{-i}) > 1/2$ for $m\geq 2$. 

(iv) By \cite[(7.3)]{DJ2}, we may assume that $\si$ is $\ell$-regular. Then there is an $\ell$-singular element $\tau$ of degree 
$e$ over $\FQ$ with the $\ell'$-part equal to $\sigma$ \cite[2.3]{DJ1}. Choose $\mu \vdash n/e$ such that
$\mu' = (1/e)\la'$. Then $L(\sigma,\la) = L(\tau,\mu)$, cf. \cite[(7.3)]{DJ2}. Since 
$e > 1$, we can apply (iii) to $L(\tau,\mu)$.  
\end{proof}

\begin{Lemma}\label{dim2}
Let $a>1$, $V = L(\si_{1},\la^{(1)}) \circ  \ldots \circ L(\si_{a},\la^{(a)})\in\IBr(GL_n)$, and $V^{U} = 0$. Then $\dim V > q^{n^{2}/4 + 2}(q-1)$ if $n \geq 5$, and $\dim V>q^5(q-1)$ if $n = 4$. 
\end{Lemma}
\begin{proof}
Set $V_{i} = L(\si_{i},\la^{(i)})$, $d_i=\deg(\si_{i})$, $n_{i} = d_i  |\la^{(i)}|$. By Theorem~\ref{hom2}, either $d_i > 1$, or $\la^{(i)}$ is $e(d_i)$-divisible, whence 
$\dim V_{i} > q^{n_{i}^{2}/4 - 1}(q-1)/2$ by Lemma~\ref{dim1}. Moreover, we have $\dim V_i\geq q-1$  if $n_i=2$, and $\dim V_i\geq (q-1)(q^2-1)$ if $n_i=3$. Indeed, by Theorem~\ref{TQ}, $V_i^{U_i}=0$ where $U_i$ is the unipotent radical of a maximal parabolic in $GL_{n_i}$. Now the bounds follow using the explicit description of $\IBr(GL_2)$ and $\IBr(GL_3)$ in \cite{J}. 

We apply induction on $a \geq 2$. 
Let $a = 2$. If $n = 4$, then $n_{1} = n_{2} = 2$, and
$$\textstyle \dim V \geq (q-1)^{2} \cdot \frac{(q^{4}-1)(q^{3}-1)}{(q^{2}-1)(q-1)} > q^{n^{2}/4 + 1}(q-1).$$
If $n = 5$, then we may assume $n_{1} = 2$, $n_{2} = 3$, and so
$$\textstyle \dim V \geq (q-1)^{2}(q^{2}-1) \cdot \frac{(q^{5}-1)(q^{4}-1)}{(q^{2}-1)(q-1)} > 
  q^{n^{2}/4 + 2}(q-1).$$ 
A similar calculation shows that $\dim V > q^{n^{2}/4 + 2}(q-1)$ if $n \geq 6$ but
$n_{1}n_{2} < 12$. Thus we may assume $n_{1}n_{2} \geq 12$. Now 
\begin{eqnarray*}
\dim V &=& \dim V_{1}\cdot \dim V_{2}  \cdot \frac{\prod^{n_{1}+n_{2}}_{i=1}(q^{i}-1)}
  {\prod^{n_{1}}_{i=1}(q^{i}-1) \cdot \prod^{n_{2}}_{i=1}(q^{i}-1)}\\
  & \geq& \left(\textstyle{\frac{q-1}{2}}\right)^{2} \cdot q^{n_{1}^{2}/4 + n_{2}^{2}/4 + n_{1}n_{2} -2}
  \\
  &=&  \left(\textstyle{\frac{q-1}{2}}\right)^{2} \cdot q^{n^{2}/4 + n_{1}n_{2}/2-2} 
  \geq (q-1)q^{n^{2}/4+2}.
\end{eqnarray*}

For the induction step, let $a\geq 3$ and $V' := V_{1} \circ \ldots \circ V_{a-1}$. 
Then $m:= n-n_{a} \geq 4$. By inductive assumption, 
$\dim V' \geq q^{m^{2}/4 + 1}(q-1)$. Recall that 
$\dim V_{a} \geq q^{n_{a}^{2}/4 - 1}(q-1)/2$, and $mn_{a} \geq 2(n-2) \geq 8$ as 
$n \geq 6$. Hence
\begin{eqnarray*}
\dim V &=& \dim V'\cdot\dim V_a \cdot \textstyle\frac{\prod^{m+n_{a}}_{i=1}(q^{i}-1)}
  {\prod^{m}_{i=1}(q^{i}-1) \cdot \prod^{n_{a}}_{i=1}(q^{i}-1)} \\
  &\geq& \textstyle{\frac{(q-1)^{2}}{2}} \cdot q^{m^{2}/4 + n_{a}^{2}/4 + mn_{a}}
   \geq q^{n^{2}/4 + mn_{a}/2-1} \geq q^{n^{2}/4+3}.  
\end{eqnarray*}
This completes the proof.
\end{proof}

\begin{Corollary}\label{dim3}
Let $n \geq 4$, $W \in \IBr(GL_{n})$, $SL_{n} \leq G \leq GL_{n}$, and 
$V$ be an irreducible constituent of $W{\dar}_{G}$. Then one of the following holds.
\begin{enumerate}
\item[{\rm (i)}] $\dim\End_{P_{1} \cap G}(V) \geq 2$. 

\item[{\rm (ii)}] $\dim V  > q^{n^{2}/4 + 2}$ if $n \neq 4,6$, $\dim V \geq (q^{2}-1)(q^{3}-1)$ if
$n = 4$, and $\dim V \geq (q^{2}-1)(q^{4}-1)(q^{5}-1)$ if $n = 6$.

\item[{\rm (iii)}] $W \cong L(\si,(n/2))$ for some element $\si$ of degree $2$.
\end{enumerate}
\end{Corollary}

\begin{proof}
Note that
$\dim V \geq \dim W/(q-1)$.  
If $\dim\End_{P_{1} \cap G}(V) = 1$ then $\dim\End_{P_{n-1} \cap G}(V^\tau) = 1$, see Remark~\ref{p-hp},  and so $(W^\tau)^{U} = 0$. If 
$W^\tau$ is Harish-Chandra induced, we can apply Lemma~\ref{dim2}. Otherwise 
$W^\tau = L(\si^{-1},(n/d))$ for some element $\si$ of degree $d > 1$ by Theorem~\ref{hom22}. By Remark~\ref{p-hp}, this is equivalent to $W\cong L(\si,(n/d))$. If (iii) fails then $d > 2$.
If $n = 4$, then $d = 4$, whence $\dim W  \geq (q-1)(q^{2}-1)(q^{3}-1)$.
If $n = 5$, then $d = 5$, whence $\dim W = \prod^{4}_{i=1}(q^{i}-1) > q^{n^{2}/4 + 2}(q-1).
$
The case $n = 6$ can be checked similarly. If $n \geq 7$, then Lemma~\ref{dim1}(iii) yields 
$\dim W \geq (q-1)q^{n^{2}/3 -1}/2 > (q-1)q^{n^{2}/4 + 2}.$ 
\end{proof}

\subsection{$\mathbf{SL_n}$-reducibility and a dimension bound}
We show that 
$GL_n$-modules which are reducible over $SL_n$ normally have large dimension:

\begin{Proposition}\label{dim4}
Let $n\geq 4$ be even, and $V$ be an irreducible $\FF GL_n$-module which is reducible over 
$SL_{n}$. Then one of the following holds:
\begin{enumerate}
\item[{\rm (i)}] $\dim V > q^{(n^{2}+n)/4}(q-1)$.

\item[{\rm (ii)}] $V \cong L(\si,(n/2))$ for some element $\si$ of degree $2$.

\item[{\rm (iii)}] $\ell\neq 2$ and $V \cong L(\si,(n/2)) \circ L(-\si,(n/2))$ for an $\ell'$-element $\si\in\FF_q^\times$. 

\item[{\rm (iv)}] $n = 4$ and $\dim V  \geq (q-1)(q^{2}-1)(q^{3}-1)$.
\end{enumerate}
Furthermore, if (ii) or (iii) occurs, then $V\dar_{SL_n}$ is a sum of two irreducible constituents $V^{1}$ and $V^{2}$ which lift to  $\CC SL_n$-modules $V_\CC^{1}$ and $V_\CC^{2}$, respectively, and have  the unique subgroup  
in $GL_n$ of index $2$ as their inertia group.
\end{Proposition}

\begin{proof}
If $m\geq 2$ we call an irreducible $\FF GL_m$-module $X$ {\em large} if  
\begin{equation*}
\dim X\geq  \left\{ \begin{array}{ll}q-1, & \mbox{if $m = 2$},\\(q-1)(q^{m-1}-1), & \mbox{if $m = 3,4$},\\
    q^{m^{2}/4-1}(q-1)/2, & \mbox{if $m \geq 5$}. \end{array}\right.
\end{equation*}

1) Let $V = V_{1} \circ \ldots \circ V_{a}$ for some (not necessarily irreducible) 
$GL_{n_i}$-modules $V_i$, with $n_i\geq 2$ for all $i$.
We claim that (i) holds if $a\geq 2$ and the $GL_{n_i}$-modules $V_i$ are large for all $i$. Apply induction on $a \geq 2$. Let $a = 2$. If $n = 4$, then $n_{1} = n_{2} = 2$, and  
$\textstyle\dim V \geq (q-1)^{2} \cdot \frac{(q^{4}-1)(q^{3}-1)}{(q^{2}-1)(q-1)} > (q-1)q^{(n^{2}+n)/4}.$
The case $n = 6$ (when $n_{1} = 2$, $n_{2} = 4$ or $n_{1} = n_{2} = 3$) is similar.
If $n \geq 8$ then $n_{1}n_{2} \geq 2(n-2) \geq 8+n/2$. As  $\dim V_{i} \geq q^{n_{i}^{2}/4-1}(q-1)/2$, we have 
\begin{eqnarray*}
\dim V &=& \dim V_{1}\cdot \dim V_{2}  \cdot \textstyle{\frac{\prod^{n_{1}+n_{2}}_{i=1}(q^{i}-1)}
  {\prod^{n_{1}}_{i=1}(q^{i}-1) \cdot \prod^{n_{2}}_{i=1}(q^{i}-1)}} \\
&\geq& \left((q-1)/2)\right)^{2} \cdot q^{n_{1}^{2}/4 + n_{2}^{2}/4 + n_{1}n_{2} -2}\\
  &\geq&  (q-1)q^{n^{2}/4 + n_{1}n_{2}/2-4} 
  \geq (q-1)q^{(n^{2}+n)/4}.
  \end{eqnarray*} 
We have completed the base case $a=2$. 
The induction
step follows by considering $V = W \circ V_{a}$ similarly to the case $n\geq 8$ above. 

2) Let $W = L(\si,\la) \in \IBr(GL_m)$ and either $d:=\deg(\si) \geq 2$ or $\la$ be  
$e$-divisible. Then 
Lemma~\ref{dim1}(iii),(iv) (and a direct check for $m \leq 4$) show that $W$ is large. If in addition $2|m$ and $d \geq 2$, then $W$ satisfies (i), (ii), or (iv) for $V=W$. Indeed, if $d \geq 3$ and $m > 6$, then by Lemma~\ref{dim1}(iii), 
$\dim W > (q-1)q^{m^{2}/3-2} > (q-1)q^{(m^{2}+m)/4}$. The case $d \geq 3$, $m = 6$ is checked 
directly. Let $d = 2$ but $\la \neq (m/2)$. 
By \cite[5.5d]{BDK} (or \cite{BK}), 
$$
\textstyle \dim W > q^{2(\frac{m}{2}-1)}\prod^{m/2}_{i=1}(q^{2i-1}-1)> (q-1)q^{\frac{m^{2}}{4} + m-4} > (q-1)q^{(m^{2}+m)/4}
$$
when $m \geq 6$. If $m = 4$, $W$ satisfies (iv).    

3) Let $j\geq 2$, $W  = L(\tau_{1},\la) \circ \ldots \circ L(\tau_{j},\la)\in \IBr(GL_m)$, and $k:=\deg(\tau_{i}) \cdot |\la|$ be the same for all $i$. We claim that $W$ is large and, if $2{|}m$, then either $W$ satisfies (i) for $V=W$ or 
$W \cong L(\si,(m/2)) \circ L(\tau,(m/2))$ for  $\ell'$-elements $\si,\tau$ 
of degree $1$. Indeed, 
$\dim W = \prod^{j}_{i=1}\dim L(\tau_{i},\la) \cdot \prod^{kj}_{i=1}(q^{i}-1)/
  (\prod^{k}_{i=1}(q^{i}-1))^{j}.$
The cases $2 \leq m \leq 4$ are checked directly. For $m \geq 5$,  
$\dim W > q^{\sum^{kj}_{i=1}i - j\sum^{k}_{i=1}i} = q^{k^{2}j(j-1)/2} \geq q^{m^{2}/4},$
so $W$ is large. Let $2{|}m$. If $j \geq 3$ then 
$\dim W > q^{k^{2}j(j-1)/2} \geq q^{m^{2}/3} >q^{(m^{2}+m)/4 +1}$, as $m \geq 6$. 
If $j = 2$ but $\la \neq (m/2)$, then $\dim L(\tau_{i},\la) > q^{m/2-1}$ \cite{GT1,BK}, whence
$\dim W > q^{m^{2}/4+m-2} > q^{(m^{2}+m)/4+1}$. 

4) Let $SL_{n} \leq R,T \leq G$ be such that 
$R/SL_n = O_{\ell'}(G/S)$, $T/SL_n = O_{\ell}(G/S)$. As $V\dar_{SL_n}$ is reducible, 
so is $V\dar_R$ or $V\dar_T$   
by Lemma~\ref{rami}(i). Write $V = L(\si_{1},\la^{(1)}) \circ \ldots \circ L(\si_{a},\la^{(a)})$ for $\ell'$-elements $\si_{i}$ with $\deg(\si_i)=d_i$. 

5) Let $V\dar_{R}$ is reducible. By Lemma~\ref{rami}(i), $\kappa^{GL_n}_{SL_n}(V)_\ell>1$. By Theorem~\ref{TMainSL}, $e=\ell$ and $\la^{(1)}, \ldots ,\la^{(a)}$ are
$\ell$-divisible. If $a \geq 2$, then (i) holds by 1) and 2). Let $a = 1$. As $\ell|(q-1)$, there is an 
$\ell$-singular element $\tau$ of degree $\ell d_{1}$ and a partition $\mu$ such that 
$V = L(\si_{1},\la^{(1)}) \cong L(\tau,\mu)$, see \cite[
2.3]{DJ1}, \cite[(7.3)]{DJ2}.  In fact, $\mu' = (\la^{(1)})'/\ell$. 
If 
$\ell d_{1} \geq 3$, the 
results of 2) applied to $L(\tau,\mu)$ imply that (i) or (iv) holds. Let $\ell d_{1} =2$. Then $\ell = 2$ and $d_{1} = 1$.
Applying 2) to $L(\tau,\mu)$, we see that (i), (ii) or (iv) holds. Assume (ii) happens. Then $\la^{(1)} = (n/2,n/2)$. As $d_{1} = 1$, we may assume that $V = L(1,(n/2,n/2))$. 
We may choose $\tau\in\FF_{q^2}^\times$ to be a $2$-element with $\tau^{q-1}=-1$. 
Then $V$ lifts to the complex module $V_\CC=L_\CC(\tau,(n/2))$.
By Theorem~\ref{TMainSL}, both $V\dar_{SL_n}$ and $(V_\CC)\dar_{SL_n}$ are direct sums of 
two irreducible constituents. 

6) It remains to consider the case where 
$V{\dar}_{T}$ is reducible. 
By Lemmas~\ref{linear}, \ref{tensor}, there is 
an $\ell'$-element $\tau \neq 1$ of degree $1$ such that 
\begin{equation}\label{iso}
  L(\si_{1},\la^{(1)}) \circ \ldots \circ L(\si_{a},\la^{(a)}) \cong 
    L(\tau\si_{1},\la^{(1)}) \circ \ldots \circ L(\tau\si_{a},\la^{(a)}).
\end{equation}
If $a \geq 2$ and $d_{i} \geq 2$ for all $i$, then (i) holds by 1) and 2). Assume that  
$a \geq 2$, $d_{i} = 1$ for $1 \leq i \leq k \leq a$, and $d_{i} > 1$ for 
$k < i \leq a$. Then by (\ref{iso}), the set $\{(\si_{i},\la^{(i)})\mid 1 \leq i \leq k\}$ is stable under multiplication by $\tau$, 
so it partitions into orbits 
of the form 
$\{ (x,\mu),(\tau x,\mu), \ldots ,(\tau^{j-1}x,\mu) \}$. 
If $k < a$ or
$k = a$ but we have at least two orbits, then (i) holds by 1), 2), 3). If $k = a$ and we have one orbit, but (i) fails, then by 3), $V\cong L(\si_1,(n/2))\circ L(\si_2,(n/2))$. As $\{\si_{1},\si_{2}\}$ is stable under multiplication by $\tau  \neq 1$, we have $\tau  = -1$, $\si_{2} = -\si_{1}$, and $\ell\neq 2$. Note that in this case $V$ lifts to the complex 
module $V_\CC = L_\CC(\si_{1},(n/2)) \circ L_\CC(-\si_{1},(n/2)) $. Furthermore, by Theorem~\ref{TMainSL}, both $V\dar_{SL_n}$ and $(V_\CC)\dar_{SL_n}$ are direct sums of 
two irreducible constituents. 

Finally, let $a = 1$. By (\ref{iso}), $\si_{1}$ and $\tau \si_{1}$ are 
conjugate,  so $d_{1} > 1$. Applying the results of 2), we see that (i), (ii), or
(iv) holds. Assume (ii) holds. As $d_{1} = 2$, we must have 
$\si_{1}^{q} = \tau \si_{1}$ and so again $\tau  = -1$. Also, $V$ lifts to the complex 
module $V_\CC = L_\CC(\si_{1},(n/2))$. Furthermore, by Theorem~\ref{TMainSL}, both $V\dar_{SL_n}$ and $V_\CC\dar_{SL_n}$ are direct sums of 
two irreducible constituents.

7) If  (ii) or (iii) hold, we have shown that $V$ lifts to a complex module $V_\CC$, and both $V$ and $V_\CC$ are sums of two irreducible $SL_n$-constituents. As $GL_n/SL_n$ is cyclic, these constituents
must have the same inertia group in $GL_n$, which is the unique subgroup of index $2$ in $GL_n$.  
\end{proof}

\subsection{Some special modules}

\begin{Lemma}\label{L2M}
Let $q=p^f$, $2{|}n\geq 6$, $SL_n\leq G\leq GL_n$, $H<G$, 
$W = L(\tau,(n/2)) \circ L(-\tau,(n/2)))$ for $\tau$ of degree
$1$ or  
$W = L(\si,(n/2))$ for $\si$ of degree
$2$, and 
$V\in\IBr(G)$ be an irreducible constituent of $W{\dar}_G$. Then one of the following holds
\begin{enumerate}
\item[{\rm (i)}] $V\dar_H$ is reducible;
\item[{\rm (ii)}] $|H/Z(H)| > q^{n^{2}/2-4}$ 
and $|H|$ is divisible by a p.p.d. $r$ for $(p,(n-1)f)$. Moreover, if $q$ is a square then $|H|$ is divisible by a p.p.d. $s$ for $(p,(n-1)f/2)$. 
\end{enumerate}
\end{Lemma}
\begin{proof}
Let (i) fail. 
By Lemma~\ref{dim1}, $\textstyle\dim W \geq \prod^{n/2}_{i=1}(q^{i}-1) > q^{n^{2}/4-1}(q-1)/2$, $\dim W$ is 
divisible by a p.p.d. $r$ for $(p,(n-1)f)$
and by a p.p.d. $s$ for $(p,(n-1)f/2)$ if $q$ is a square. 
If $W\dar_{G}$ is 
irreducible then $V = W\dar_{G}$ lifts to a complex module $\WN$. 
As $V\dar_{H}$ is irreducible, so is $\WN\dar_{H}$, whence $|H/Z(H)| > q^{n^{2}/2-4}$ 
and $|H|$ is divisible by $r$ and by $s$ if $q$ is a square. 

Now suppose that $W\dar_{G}$ is reducible. In particular, $q \geq 3$ and 
$W\dar_{SL_n}$ is reducible. By Lemma~\ref{dim1},  $\dim W < q^{(n^{2}+n)/4}(q-1)$. Hence by Proposition
\ref{dim4}, $W\dar_{SL_n}=W^{1}\oplus W^{2}$, where each $W^{i}$ lifts to a complex module $\WN^{i}$, and $W^{i}$ and 
$\WN^{i}$ have the same inertia group $I$, which is the unique subgroup of index $2$ in $GL_n$.
We may assume that $V\dar_{SL_n} \cong W^{1}$. Thus $G$ stabilizes $W^{1}$, and 
so $G \leq I$. But $G/SL_n$ is cyclic, hence the $G$-invariance of $\WN^{1}$ implies that 
$\WN^{1}$ extends to a $\CC G$-module $\VC$. Now $V$ and 
$(\VC \mod \ell)\dar_{SL_n}$ are two extensions to $G$ of $W^{1}$, and again $G/SL_n$ is cyclic.
Therefore, $(\VC \mod \ell) \cong V \otimes A$ for some $A \in \IBr(G/SL_n)$, cf.
\cite[
III.2.14]{F}. Recall $V\dar_{H}$ is irreducible and $A$ is one-dimensional.
It follows that $(\VC \mod \ell)\dar_{H}$ is irreducible; in particular, $\VC$ is irreducible
over $H$. Notice that $\dim \VC = (\dim W)/2$ is divisible by $r$ and is greater than
$q^{n^{2}/4-2}$. We conclude that $|H/Z(H)| > q^{n^{2}/2-4}$ 
and $|H|$ is divisible by $r$ and by $s$ if $q$ is a square. 
\end{proof}

\section{Restricting to arbitrary parabolic subgroups}
\subsection{Some general facts}
\begin{Lemma}\label{div1}
Let $N\lhd H$ with $H/N$ cyclic, $P\leq H$ be such that
$PN = H$, $W \in \IBr(H)$ lift to a complex module $W_\CC$, $V$ an irreducible constituent of $W{\dar}_N$, and 
$V\dar_{P \cap N}$ be irreducible. 
Then $\dim W$ divides $|P \cap N|\cdot(H:N)$.
\end{Lemma}

\begin{proof}
As $H/N$ is cyclic and $PN = H$, we have $W\dar_{N} = \oplus^{t}_{i=1}g_iV$
with $g_{i} \in P$, and $t{|}(H:N)$. Since $V{\dar}_{P \cap N}$ is irreducible, so is 
each $g_iV{\dar}_{P \cap N}$, as $P$ normalizes $P \cap N$. So all composition factors  of 
$W\dar_{P \cap N}$ have dimension $(\dim W)/t$. If $X$ is a composition factor of 
$(W_\CC)\dar_{P \cap N}$ then $\dim X = k(\dim W)/t$ for some $k\in\NN$. As 
$\dim X$ divides $|P \cap N|$, the claim follows.
\end{proof} 

\begin{Lemma}\label{div2}
Let $N\lhd H$ with $H/N$ cyclic of order $d$. 
Assume that $P, Q\leq H$  are such that $QN = H$, 
$Pg_{1}Q \neq Pg_{2}Q$ for some $g_{1},g_{2} \in H$, and $W \in \IBr(H)$ is 
$U\uar^{H}$ for some $\FF P$-module $U$. Assume in addition that 
$\gcd(K_{1},K_{2}) \leq \max\{K_{1}/d,K_{2}/d\}$, where  
$K_{i} = (Q:P^{g_{i}} \cap Q)$ for $i = 1,2$. Then any $N$-irreducible constituent $V$ of 
$W\dar_{N}$ is reducible over $Q \cap N$.
\end{Lemma}

\begin{proof}
Assume that $V\dar_{Q \cap N}$ is irreducible. As $H/N$ is cyclic and 
$QN = H$, we have $W\dar_{N} = \oplus^{t}_{i=1}V_{i}$,
with $V_{i} = g_{i}V$ for some $g_{i} \in Q$, and $t{|}d$. Since $Q$ normalizes $Q \cap N$, every $V_{i}\dar_{Q \cap N}$ is irreducible,
of dimension $D := \dim V = (\dim W)/t$. On the other hand, by Mackey's Theorem,
$$W\dar_{Q} = \bigoplus_{PxQ \in P \backslash G/Q}
 (x^{-1}U\dar_{P^{x} \cap Q})\uar^{Q}$$ contains the summands 
$(g_{i}^{-1}U\dar_{P^{g_{i}} \cap Q})\uar^{Q}$ of dimension $K_{i}\dim U$ for
$i = 1,2$. Hence $D$ divides $K_{1}\dim U$ and $K_{2}\dim U$, and so $D{|}\gcd(K_{1},K_{2})\dim U$. By assumption, there is $j\in\{1,2\}$ such that 
$\gcd(K_{1},K_{2}) \leq K_{j}/d$. Hence 
$dD \leq K_{j}\dim U < \dim W=tD$, contradicting $t{|}d$.   
\end{proof}

\subsection{Irreducible restrictions for parabolics}
\begin{Lemma}\label{p1p2}
Let $n\geq 2$, and $P,Q<GL_n$ be maximal parabolic subgroups. Then there exist $g_{1},g_{2} \in GL_n$ such that 
$Pg_{1}Q \neq Pg_{2}Q$ and $K_{2}/K_{1} > q-1$,
where $K_{i} = (Q:P^{g_{i}} \cap Q)$ for $i = 1,2$.
\end{Lemma}

\begin{proof}
Let $e_{1}, \ldots ,e_{n}$ be a basis of 
$\NC$. We may assume 
that $Q$ is the stabilizer of an $m$-space $\langle e_{1}, \ldots ,e_{m}\rangle$,  
$1 \leq m \leq n-1$. We may replace $P$ by any of its conjugates. So we
may assume that $P=\Stab_{GL_n}(\langle e_{1}, \ldots ,e_{k}\rangle)$,  $1 \leq k \leq n-1$. Setting 
$L_{i} := |P^{g_{i}} \cap Q|$, it suffices to find $g_{1}$, $g_{2}$ such that
$L_{1}/L_{2} > q-1$.

First we consider the case $k = m$. Choosing $g_{1} = 1$, we get 
$\textstyle L_{1} = q^{n(n-1)/2}\prod^{k}_{i=1}(q^{i}-1) \prod^{n-k}_{i=1}(q^{i}-1).$ 
There is $g_{2} \in GL_n$ such that 
$P^{g_{2}} = \Stab_{GL_n}(\langle e_{1}, \ldots ,e_{k-1},e_{k+1}\rangle)$. Then 
$\textstyle L_{2} = (q-1)^{2}q^{n(n-1)/2-1}\prod^{k-1}_{i=1}(q^{i}-1) \prod^{n-k-1}_{i=1}(q^{i}-1),$
whence $L_{1}/L_{2} = q(q^{k}-1)(q^{n-k}-1)/(q-1)^{2} \geq q$. 

Next we consider the case $k < m$. Choosing $g_{1} = 1$, we get 
$$\textstyle L_{1} = q^{n(n-1)/2}\prod^{k}_{i=1}(q^{i}-1)
  \prod^{m-k}_{i=1}(q^{i}-1) \prod^{n-m}_{i=1}(q^{i}-1).$$
There is $g_{2} \in GL_n$ such that 
$P^{g_{2}} = \Stab_{GL_n}(\langle e_{1}, \ldots ,e_{k-1},e_{m+1}\rangle)$. Then 
$$\textstyle 
L_{2} = (q-1)q^{n(n-1)/2-(m-k+1)}\prod^{k-1}_{i=1}(q^{i}-1) \prod^{m-k+1}_{i=1}(q^{i}-1)
   \prod^{n-m-1}_{i=1}(q^{i}-1),$$ 
whence $\textstyle
L_{1}/L_{2} = \frac{q^{m-k+1}(q^{k}-1)(q^{n-m}-1)}{(q^{m-k+1}-1)(q-1)} > q-1.$

Finally, consider the case $k > m$. According to the previous case, there is 
$h_{2} \in GL_n$ such that $QP \neq Qh_{2}P$ and 
$|Q \cap P|/|Q^{h_{2}} \cap P| > q-1$. Choosing $g_{2} = h_{2}^{-1}$, we get 
$PQ \neq Pg_{2}Q$ and 
$|P \cap Q|/|Q \cap P^{g_{2}}| > q-1$, as required.
\end{proof}

Now we are able to prove the main result of this section:

\begin{Theorem}\label{para}
Let $n\geq 2$, $SL_{n} \leq G \leq GL_{n}$, $V \in \IBr(G)$ be a constituent of $W{\dar}_G$ for $W\in \IBr(GL_n)$ with $\dim V > 1$, and $P$ be a proper parabolic 
subgroup of $GL_n$ such that $V\dar_{P \cap G}$ is irreducible. Then $P$ is the stabilizer 
in $GL_n$ of a $1$-space or of an $(n-1)$-space in the natural $GL_n$-module $\NC$,
and $W = L(\tau,(k))$ for some element $\tau$ of degree $n/k > 1$.
\end{Theorem}

\begin{proof}
First we show that $W$ is Harish-Chandra indecomposable. Otherwise $W = U\uar_Q^{GL_n}$
for some $\FF Q$-module $U$ and a proper parabolic $Q$. We may assume that $Q$ and $P$ are both maximal in $GL_n$. Note that $GL_n/G$ is cyclic
of order $\leq q-1$ and $PG = GL_n$. By Lemma~\ref{p1p2}, $(GL_n,G,Q,P)$ satisfies all 
the hypotheses of Lemma~\ref{div2} on $(H,N,Q,P)$. Hence  $V\dar_{P \cap G}$ is 
reducible.

Thus $V = L(\si,\la)$ for some $\la$ and $\sigma$ of 
degree $n/|\la|$. 
For $n=2$ the theorem is checked directly. Let $n\geq 3$. 
If $\dim\End_{P_{1} \cap G}(V) > 1$, the irreducibility of 
$V\dar_{P \cap G}$ implies by Corollary~\ref{CRed} that $P \cap G$ is transitive on $1$-spaces in $\NC$, a contradiction. So 
$\dim\End_{P_{1} \cap G}(V) = 1$. By Theorem~\ref{hom22},
$V = L(\tau,(n/d))$ for some element $\tau$ of degree $d$. By Lemma~\ref{dim1}(i), $\dim L(\tau,(n/d))=\prod_{i=1}^n(q^i-1)/\prod_{i=1}^{n/d}(q^{di}-1)$. We have $d>1$ as $\dim W > 1$. 

By Lemma~\ref{div1}, $\dim W$ divides $(q-1)|P \cap G|$. So $(q^{n-1}-1)/(q-1)$ divides $|P|$, see Lemma~\ref{dim1}(ii). 
Suppose that there exists a p.p.d. $r$ for $(q,n-1)$, see \S\ref{SMoreNot}. Then
the divisibility of $|P|$ by $r$ implies that $P$ cannot stabilize any $m$-space of $\NC$ for $2 \leq m \leq n-2$. 
So $P$ is the stabilizer 
of an $1$-space or of an $(n-1)$-space. By Lemma~\ref{LPPD}, it remains to
consider the cases 
$(n,q) = (7,2)$, or $n = 3$ and $q = 2^{s}-1$ is a Mersenne prime. In the former case, 
$GL_n = G$, 
 $W=V=L(\tau, (1))$, and $\dim W$ divides $|P|$. In particular $31$ divides $|P|$, so $P$ cannot stabilize an $m$-space for $m=3,4$. On the other hand the stabilizer of any $2$- or $5$-space has order not divisible by $\dim W$. We conclude that $P$ is the stabilizer of a $1$- or a $6$-space. 
Finally, let $n = 3$ and $q = 2^{s}-1$ is a Mersenne prime. If $P$ 
does not satisfy the conclusion of the theorem then $|P| = q^{3}(q-1)^{3}$ and 
$\dim W = (q-1)(q^{2}-1)$. We have shown that $(\dim W)/(q-1)$ divides $|P \cap G|$,
so $2^{s}|(2^{s}-2)^{2}$. This can happen only when $s = 2$, whence $q = 3$. But in this
last situation $GL_n = Z(GL_n)G$, $W=V$, and $\dim W$ does not divide $|P|$. 
\end{proof}

\subsection{Subgroups of $\mathbf{P_{n-1}}$}\label{SSPar} 
Let $SL_{n} \leq G \leq GL_{n}$, $V \in \IBr(G)$ with $\dim V> 1$, and 
$P = P_{n-1}=LU$ be the stabilizer in $GL_{n}$ of an $(n-1)$-space.
In this subsection we study irreducible restrictions $V{\dar}_H$ for subgroups $H \leq P \cap G$. By Theorem \ref{para}, we may (and will) assume that 
$V$ is a constituent of $W\dar_{G}$ for $W$ of the form $L(\si,(k))$, with $\si$ of degree $d=n/k\geq 2$, whose dimension is given in Lemma~\ref{dim1}(i). Set $Z = Z(GL_{n})$. Then $P = Z \times AGL_{n-1}$. Also $AGL_{n-1}= UL_{1}$, where 
 $L_{1} \cong GL_{n-1}$. Clearly, $V\dar_{H}$ is irreducible if and only if 
$V$ is irreducible over $HZ$, and $HZ = Z \times H_{1}$ for a suitable subgroup 
$H_{1} \leq AGL_{n-1}$. So we may (and will) assume that $H \leq AGL_{n-1}$. We will also identify $AGL_{n-1}/U$ with $L_1$. Then $HU/U$ is identified with a subgroup of $L_1$. The assumptions we have made so far will be valid {\em throughout \S\ref{SSPar}.} 

\begin{Lemma}\label{Hbar}
Let $n\geq 4$ and $V{\dar}_{H}$ be irreducible. Then one of the following holds:
\begin{enumerate}
\item[{\rm (i)}] $HU/U \geq [L_{1},L_{1}] \cong SL_{n-1}$;
\item[{\rm (ii)}] $n=4$, $q=2$ or $3$, $d=2$, $W=L(\si,(2))$, $\dim V=q^3-1$, $\si^2=-1$ if $q=3$ and $\ell\neq 2$,  
 $H=UM$ where $GL_1(q^3)\lhd M\leq \Ga L_1(q^3)<L_1$. In this case  $V\dar_H$ is indeed irreducible. 
\end{enumerate}
\end{Lemma}

\begin{proof}
As $\dim(V) > 1$, 
$V\dar_{U}$ affords some, and so all nontrivial linear characters of $U$. 
Now $V\dar_{HU}$ is irreducible and $U \lhd HU$, so $HU$ must 
act transitively on $\IBr(U) \setminus \{1_{U}\}$, with $U$ acting trivially. Identifying
$\IBr(U)$ with $\FQ^{n-1}$, we see that $HU/U$ is a subgroup of $GL_{n-1}(q)$ which acts 
transitively on the nonzero vectors of $\FQ^{n-1}$. By Proposition 
\ref{step1}, one of the following holds: 
\begin{enumerate}
\item[(a)] $HU/U \rhd  SL_{a}(q_{1})$ with $q_{1}^{a} = q^{n-1}$ and $2 \leq a |(n-1)$;

\item[(b)] $HU/U \rhd  Sp_{2a}(q_{1})'$ with $q_{1}^{2a} = q^{n-1}$ and $2 \leq a|(n-1)/2$;

\item[(c)] $HU/U \rhd  G_{2}(q_{1})'$ with $q_{1}^{6} = q^{n}$, $2|q$, and $6|(n-1)$;

\item[(d)] $HU/U$ is contained in $\Gamma L_{1}(q^{n-1})$;

\item[(e)] $(q^{n-1},HU/U)$ is one of the following: $(3^{4}, \leq 2^{1+4}_{-}\cdot \SSS_{5})$,  $(3^{4}, \rhd SL_{2}(5))$,  
$(2^{4},\AAA_{7})$ or $(3^{6},SL_{2}(13))$.
\end{enumerate}
We need to show that (a) holds with $q_1=q$. Assume otherwise. 

First, we exclude (e). If $(q^{n-1},HU/U)=(2^{4},\AAA_{7})$, 
then $|H| \leq |U|\cdot |\AAA_{7}| < 315^{2} = (\dim V)^{2}$, so $V\dar_{H}$ is reducible. 
The case $(q^{n-1},HU/U)= (3^{6},SL_{2}(13))$ is similar.
Let $(n,q) = (5,3)$. As $GL_{5}(3) = Z \times SL_{5}(3)$, 
$V = W$ lifts to a complex module. So, if $V\dar_{H}$ is irreducible, then 
 $\dim V = \prod^{4}_{i=1}(3^{i}-1)$ divides $|H|$, hence $5 \cdot 13$ divides $|HU/U|$. 
However, no proper subgroup of $SL_{4}(3)$ has order divisible by $5 \cdot
13$ \cite{Atlas}. 

Assume that $n = 4$ and let $t=\kappa^{GL_4}_G(W)$. Then $t|(4,q-1)$, and 
$\dim V = (\dim W)/t \geq (q^3-1)(q-1)/t$. On the other hand, we have $HU/U\leq \Ga L_1(q^3)$ so $|HU|<4(q^3-1)^2$. So  $V\dar_H$ is reducible if $t<q-1$. Let $t=q-1$, and so $q=2,3$ or $5$. If $d=4$ then $\dim V = (q^{2}-1)(q^3-1)>2(q^3-1)$, and $V\dar_H$ is reducible. Let $d=2$. Then $t\leq 2$, so $q=2$ or $3$. Moreover, if $q=3$, the equality $t=2$ implies $\si^2=-1$ if $\ell\neq 2$, see Theorem~\ref{TMainSL}. 
If $H\not\geq U$ then $H\cap U=\{1\}$, and $|H|=|HU/U|\leq 3(q^3-1)<(\dim V)^2$. So $H\geq U$ and $H=UM$ for some $M\leq L_1$. As $HU/U$ acts transitively on $q^3-1$ non-trivilal linear characters of $U$, we deduce $M\geq GL_1(q^3)$. Conversely, if we are in  (ii), then by dimensions, $V\dar_U$ is a sum of all non-trivial linear characters of $U$ which are permuted transitively by $M$, and so $V\dar_H$ is irreducible. 

Assume that $n = 5$. Then 
$\dim V \geq (\dim W)/(q-1) = \prod^{4}_{i=2}(q^{i}-1)$. If 
$HU/U \rhd Sp_{4}(q)$, then $|HU/U| \leq |CSp_{4}(q)| = q^{4}(q-1)(q^{2}-1)(q^{4}-1)$.
If $HU/U \rhd SL_{2}(q^{2})$, then $|HU/U| \leq |\Gamma L_{2}(q^{2})| = 2q^{2}(q^{2}-1)(q^{4}-1)$.
Also, $|\Gamma L_{1}(q^{4})| = 4(q^{4}-1)$. In all cases, $|H| < (\dim V)^{2}$ and so $V\dar_{H}$ 
is reducible.

If $n = 6$ then $HU/U \leq \Gamma L_{1}(q^{5})$,  so 
$|H| \leq 5(q^{5}-1)q^{5}$, whereas $\dim V \geq (\dim W)/(q-1) \geq (q^{3}-1)(q^{5}-1)>\sqrt{|H|}$, and $V\dar_{H}$ is reducible.

Finally, let $n \geq 7$. If $n/k \geq 3$ then
$\dim V > q^{n^{2}/3-2}$ by Lemma \ref{dim1}(iii). Hence for any $s \geq 2$,
$|\Gamma_{(n-1)/s}(q^{s})| = s|GL_{(n-1)/s}(q^{s})| < sq^{(n-1)^{2}/s} \leq 2q^{(n-1)^{2}/2} < 
  (\dim V)^{2}/|U|,$ so $HU/U$ cannot lie in $\Gamma L_{(n-1)/s}(q^{s})$. This leaves two
possibilities: $HU/U \rhd Sp_{n-1}(q)$ or $HU/U \rhd G_{2}(q)'$ (with $n = 7)$. Then $|HU/U| \leq |CSp_{n-1}(q)| < (q-1)q^{n(n-1)/2} < (\dim V)^{2}/|U|$, and again 
$V\dar_{H}$ is reducible.
If $n/k = 2$ then
$\dim V \geq (\dim W)/2 = (\prod^{n/2}_{i=1}(q^{2i-1}-1))/2$. Since $n-1$ is odd, we must have 
$HU/U \leq \Gamma L_{(n-1)/s}(q^{s})$ for some $s \geq 3$. It follows that 
$|H| < (\dim V)^{2}$, again a contradiction. 
\end{proof}

Recall notation from \S\ref{SStat}.

\begin{Proposition}\label{P11}
Let $n\geq 3$. 
\begin{enumerate}
\item[{\rm (i)}] If $H \geq [P,P]= U[L,L]$ then $V\dar_{H}$ is irreducible 
if and only if $\kappa(\si,H) = \kappa_{G}(W)$.

\item[{\rm (ii)}] Let $n\geq 4$,  $H \cap U \neq 1$, and $V\dar_H$ is irreducible. Then either  $H \geq [P,P] = U[L,L]$ or the case (ii) of Lemma~\ref{Hbar} holds. 
\end{enumerate}
\end{Proposition}

\begin{proof}
(ii) If $V\dar_{H}$ is irreducible but (ii) of Lemma~\ref{Hbar} fails. Then the lemma yields   
$H/(H \cap U) \cong HU/U \geq [L_{1},L_{1}]$. But $H \cap U \neq 1$ and $[L_{1},L_{1}]$ acts transitively on
$H \setminus \{1\}$, so $H \geq U$ and $H  \geq U[L,L] = [P,P]$. 

(i) Write 
$H = UM_{1}$ with $SL_{n-1} \leq M_{1} \leq GL_{n-1}$ and let $J\leq\FF_q^\times$ be the image of $M_{1}$ under the 
determinant map. 
We have $W\dar_{G} = \oplus^{\kappa_{G}(W)}_{i=1}g_iV$ 
for some $g_{i} \in T := \{\diag(I_{n-1},a) \mid a \in \FQ^{\times}\}$ with $g_1=1$. As $T$ normalizes $H$, $H$ acts on
each $g_iV$. It suffices now to show that 
$\kappa(\si,H)$ equals the number $t$ of the irreducible summands of $W\dar_{H}$.

By Theorem \ref{TP1}, $W\dar_{AGL_{n-1}} = e_{d-1} L(\si,(k-1))$. Recalling the explicit definition of the functor $e_{d-1}$ from \cite[\S5.1]{BDK} and using the analogue of \cite[5.1a]{BDK} for $H\cap AGL_{n-1}$, we conclude that 
$t$ equals the number of irreducible summands of $L(\si,(k-1))\dar_{M_2}$, where $M_2=M_1\cap GL_{n-d}$, and $\det(M_2)=J$. By 
Lemma \ref{link2}, the latter number is equal to 
\begin{equation}\label{fort}
  \kappa^{GL_{n-d}}_{B}(L(\si,(k-1))) \cdot \min\{
    (\kappa^{GL_{n-d}}_{SL_{n-d}}(L(\si,(k-1)))_{\ell}, |GL_{n-d}/M_2|_{\ell}\}
\end{equation}
where $M_2 \leq B \leq GL_{n-d}$ and $B/M_2 = O_{\ell}(GL_{n-d}/M_2)$. Applying Lemmas
\ref{linear}, \ref{tensor}, and Theorem \ref{TMainSL}, we see that each term in (\ref{fort}) depends
only on $\si$ and $J$, but not on $k$. For instance, 
$\kappa^{GL_{n-d}}_{SL_{n-d}}(L(\si,(k-1)))_{\ell} = 
  \gcd\left(q-1,\frac{\deg(\si)}{\deg(\si_{\ell'})}\right)_{\ell}$. 
Similarly, 
$\kappa(\si,H) = \kappa^{GL_{n}}_{M_{1}SL_{n}}(L(\si,(k)))$ also equals (\ref{fort}),
since $\det(M_{1}SL_{n}) = J$. Consequently, $t = \kappa(\si,H)$.       
\end{proof}

\begin{Example}
{\rm
Let $W = L(\si,(k))$ and  $V$ be an irreducible constituent of
$W\dar_{SL_{n}}$. Then $V$ is irreducible over the subgroup $H = ASL_{n-1}(q)$ of $AGL_{n-1}$, as 
in this case, $M_{1} = SL_{n-1}(q)$ and so $\kappa(\si,H) = \kappa^{GL_{n}}_{SL_{n}}(W)$.
} 
\end{Example}

\begin{Proposition}\label{P12}
Let $n \geq 3$ and $H \cap U = 1$. Then $V\dar_{H}$ is irreducible 
if and only if $[L,L] \leq H \leq L_{1}$(up to $G$-conjugacy), $2|n$,  and one of the following holds:
\begin{enumerate}
\item[{\rm (i)}] $G=SL_n(3)$ 
and if $\ell\neq 2$ then $\si^2=-1$. 
\item[{\rm (ii)}] $G=SL_n(2)$ 
and $\si\neq 1=\si^3$. 
\end{enumerate}
\end{Proposition}

\begin{proof}
If $n = 3$, then 
$\dim V \geq (\dim W)/(q-1) \geq q^{2}-1$, whence 
$|H| = |HU/U| \leq |GL_{2}|< (\dim V)^{2}$, and $V\dar_{H}$ is 
reducible. 
So we may assume that $n \geq 4$.
Suppose $V\dar_{H}$ is irreducible. 
By Lemma \ref{Hbar}, $H \cong HU/U \geq [L_{1},L_{1}]$.
On the other hand, $HU \leq AGL_{n-1}$, hence $HU/U \leq GL_{n-1}$. Therefore 
$K := [H,H] \cong SL_{n-1}$. Now observe that $KU$ is a subgroup of $[P,P] = ASL_{n-1}$ of
order $|K| \cdot |U| = |[P,P]|$. Hence $KU = ASL_{n-1} = U[L_{1},L_{1}]$. Thus  
$K$ and $[L_{1},L_{1}]$ are two complements to $U$ in $[P,P]$.

If $(n,q) \notin \{(5,2), (4,2)\}$, then by \cite{Sah} (for the case $q = p$) and \cite{CPS} (for the case $q > 3$), 
$H^{1}(SL_{n-1}(q),\FQ^{n-1}) = 0$, whence  
$K$ is $P$-conjugate to $[L_{1},L_{1}] = [L,L]$. 
If 
$(n,q) = (5,2)$ then $\dim V = 315 > \sqrt{|H|}$, and so this case can be excluded. 
If $(n,q)=(4,2)$ then $H\cong L_{1} = SL_{3}(2)$, so we must have $\dim V = 7$. Note  that $V$ lifts to a complex module $\VC$ which is the heart of the permutation module of $G = SL_{4}(2) \cong \AAA_{8}$ acting on $8$ letters. The irreducibility of $(\VC)\dar_{H}$ is equivalent to $H$ being $2$-transitive. One can check\footnote{We thank E. O'Brien for verifying this fact on computer.} that any such a subgroup is conjugate to $L_1$. 
Thus in any case we may assume that $K = [L_{1},L_{1}]$, and so
$H \leq N_{P}(K) \cap AGL_{n-1}=L_{1}$.

As $[L,L] \leq H \leq L_{1}$, $H$ is centralized by 
$T := \{\diag(I_{n-1},a) \mid a \in \FQ^{\times}\}$. Write $W\dar_{G} = \oplus^{t}_{i=1}g_{i}V$
for some $g_{i} \in T$. 
Then every composition factor of $W\dar_{H}$ has
dimension equal to $\dim V = (\dim W)/t$. On the other hand, $W$ lifts to a 
complex module $W_{\CC} = L_{\CC}(\si,(k))$. If $W_{1}$ is an irreducible constituent of 
$W_{\CC}\dar_{H}$, then every composition factor of $W_{1}\pmod{\ell}$ has  dimension $\dim V$, 
so $\dim W_{1} \geq \dim V = (\dim W)/t$. Thus $W_{\CC}$ has at most $t \leq q-1$ 
irreducible constituents on restriction to $H$, and so on restriction to $L_{1} = GL_{n-1}$ as well. 

By \cite[5.1e]{BDK}, we have an isomorphism of functors 
$$\res^{AGL_{n-1}}_{GL_{n-1}} ~\circ~ e_{d-1}(-) \simeq R^{GL_{n-1}}_{GL_{n-d} \times GL_{d-1}}
 \left(- \otimes \Gamma_{d-1}\right),$$
 where $R$ denotes the Harish-Chandra induction, and $\Gamma_{d-1}$ is the Gelfand-Graev
representation of $GL_{d-1}$, which has at least $(q-1)q^{d-2}$ composition factors \cite[2.5d]{BDK}. 
As $L(\si,(k))\dar_{AGL_{n-1}} = e_{d-1}(L(\si,(k-1)))$, we conclude that 
$L(\si,(k))\dar_{GL_{n-1}}$ has at least $(q-1)q^{d-2}$ composition factors.
Hence, $(q-1)q^{d-2} \leq t \leq q-1$ and so $d = 2$. In this case,
$\kappa^{GL_{n}}_{SL_{n}}(W) \leq 2$ by Theorem~\ref{TMainSL}. Hence $(q-1)q^{d-2} \leq t \leq 2$, and so
$q \leq 3$.

If $q = 2$ then $\si \neq 1 = \si^{3}$, as $d=2$. Moreover,
$H= GL_{n-1}$, and 
\begin{equation}\label{Eq=2}
V\dar_{H} = L(\si,(n/2-1)) \circ L(1,(1)), 
\end{equation}
which is 
irreducible. Indeed, this is clear if $\ell \neq 3$. If $\ell=3$, then $L(\si,(n/2-1)) \circ L(1,(1))\cong L(1,(n/2-1,n/2-1,1))$, see \cite[4.3e]{BDK}. 

Let $q = 3$. Then $t = 2$ and $G = SL_{n}(3)$.
Also $\Gamma_{1}= L(1,(1)) + L(-1,(1))$ in the Grothendieck group.
If $\ell \neq 2$ then $\si \in \FF_{9}^{\times}$ 
is an $\ell'$-element, so in the Grothendieck group of $\FF L_{1}$-modules we have 
\begin{equation}\label{Eq=3}
W\dar_{L_{1}} = 
  \left(L(\si,(n/2-1)) \circ L(1,(1))\right) + \left(L(\si,(n/2-1)) \circ L(-1,(1))\right).
 \end{equation}
 As $2|n$, we have $L_{1} = Z(L_{1}) \times [L_{1},L_{1}]$. But $L_{1} \geq H \geq [L_{1},L_{1}]$, so  $\kappa^{GL_n}_H(W)=2$.
Moreover, 
$t = 2$ implies 
$\si^{2}=-1$. Finally,if $\ell = 2$, then  
$\si \in \FF_{9}^{\times}$ is a $2$-element of degree $2$. So in the Grothendieck group 
we have $W\dar_{L_{1}} = 2L(1,(n/2-1,n/2-1)) \circ L(1,(1))=2 L(1,(n/2-1,n/2-1,1))$ by \cite[4.3e]{BDK}. Thus $\kappa^{GL_n}_{L_1}(W)=2$
Hence $V\dar_H$ is irreducible. 

The `if' part follows from (\ref{Eq=2}) and (\ref{Eq=3}). 
\end{proof}

\begin{Lemma}\label{Ln=3}
Let $n = 3$, $H \cap U \neq 1$, and $V\dar_{H}$ be irreducible. Then $H {\geq} [P,P]$.
\end{Lemma}
\begin{proof}
Assume the contrary. As $HU/U$ acts transitively on $\Irr(U) \setminus \{1_{U}\}$
and so on $U \setminus \{1\}$, we have $H \geq U$ and $H = UM$ for some
$M \leq L_{1}$. Let $t=\kappa^{GL_3}_{G}(W)$. Then $\dim V=(q-1)(q^2-1)/t$, and $t{|}(3,q-1)$. 

Let $t=1$. Then $|M|\geq(\dim V)^2/|U|>(q^2-2)(q-1)^2$. If $q>2$, then $(GL_2:M)<2$ and so $M=L_1$, a contradiction. If $q=2$, then $\dim V=3$. By the Fong-Swan theorem, $V\dar_H$ lifts to a $\CC H$-module as $H$ is solvable. So $3$ divides $|H|$ and $|M|$, which implies that $M\geq SL_2(2)$, a contradiction. 

Finally, let $t=3$ and $q\equiv 1\pmod{3}$. Then $|M|\geq(\dim V)^2/|U|>(q^2-2)(q-1)^2/9$. So, if $q\geq 13$, then $(GL_2(q):M)<q-2$. But the index of any proper subgroup of $SL_2(q)$ is at least $q+1$ (for such $q$) \cite[Table 5.2A]{KL}. Hence $M\geq SL_2(q)$, a contradiction. Let $q=7$. Then $(SL_2(7):M\cap SL_2(7))<11$. But $M$ acts transitively on $48$ non-trivial characters of $U$, so $|M\cap SL_2(7)|$ is divisible by $8$. By \cite{Atlas}, $|M\cap SL_2(7)|=48$, and so $|M|$ divides $6\cdot 48$. Recall that $V{\dar}_U$ affords all non-trivial characters with multiplicity $2$. Let $\la$ be such a character, and $J=\operatorname{Stab}_H(\la)$. Then $V\dar_H=B\uar_J^H$ for some $2$-dimensional $J$-module $B$. Again, $J$ is solvable, as $J/U$ has order dividing $6$. So $B$ lifts to a complex module of dimension $2$, on which $U$ acts via the character $2\la$. Hence $J/U$ acts irreducibly on this complex module. So $J/U\cong S_3$. On the other hand, the stabilizer of $\la$ in $GL_2(7)$ is isomorphic to $C_7:C_6$, and so it cannot contain $S_3$ as a subgroup, giving a contradiction. 
Let $q=4$. As $t=3$, we have $G=SL_3(4)$, $\dim V=15$. Note that $M\leq L_1\cap G=SL_2(4)$. As a proper subgroup of $SL_2(4)$, $M$ is solvable, and hence so is $H$.  Again, $V\dar_H$ lifts to a complex module, hence the irreducibility of $V\dar_H$ implies that $15$ divides $|H|$, and also $|M|$. But $SL_2(4)$ does not have such a subgroup, giving a contradiction. 
\end{proof}


\begin{Lemma}\label{Ln=2}
If $n=2$ then $V\dar_H$ is irreducible if and only $HZ(G)=P_1\cap G$ and one of the following holds:
\begin{enumerate}
\item[{\rm (i)}] $\dim V=q-1$ and $GZ=GL_2$; 
\item[{\rm (ii)}] $q>3$ is odd, $\dim V=(q-1)/2$, and $G=Z(G)SL_2$.
\end{enumerate}
\end{Lemma}
\begin{proof}
The cases $q=2,3$ are easy to check, so assume $q\geq 4$. Let $V\dar_H$ be irreducible. Note that $\dim V= (q-1)/t$ for $t=1$ or $2$. We have $H\cap U\neq 1$, for otherwise  $H\cong HU/U\leq AGL_1/U$ is abelian, giving a contradiction. Note that $HU$ acts transitively on the non-trivial linear characters of $U$ which appear in $V\dar_U$. These characters appear with multiplicity $1$, so their number is $(q-1)/t$. Hence $HU$ has an orbit of length $(q-1)/t$ on $U\setminus\{1\}$. One can now check that  there are $t$ $H$-orbits on $U\setminus\{1\}$, each of length $(q-1)/t$.   As $H\cap U\neq 1$, we conclude that $H\geq U$. If $t=1$ then $H=AGL_1$ because it acts transitively on $q-1$ non-trivial characters of $U$, and we arrive at (i). If $t=2$, we arrive at (ii).
The `if' part is straightforward. 
\end{proof}

\section{Transitive subgroups}

We consider the groups 
transitive on $1$-spaces from  
Proposition~\ref{step1}.

\subsection{$\bf\Gamma L_d(9)$} First, we rule out the case left in Proposition~\ref{semi2}:
 
\begin{Proposition}\label{semi3}
Let $\ell \neq 2$, $d \geq 3$ be odd, $n = 2d$, $S = SL_{n}(3)$, $H = \Gamma L_{d}(9) \cap S$, and $V \in \IBr(S)$ with $\dim V > 1$. Then $V{\dar}_H$ is reducible.
\end{Proposition}

\begin{proof}
Set $G = GL_{n}(3)$, $K=\Gamma L_d(9)$. If $V$ extends to a $G$-module then the extension is reducible over $H$ by Proposition~\ref{semi2}.  
So we may assume that there is $W \in \IBr(G)$ such that $W\dar_{S} = V \oplus V'$. 
Proposition~\ref{dim4} allows us to restrict to one of the following two cases:
(a) $\dim V>3^{d^2+d/2}$. In this case $|H|=|K|/2=|GL_d(9)|<9^{d^2}<(\dim V)^2$. So $V{\dar}_H$ is reducible. 
(b) $W$ satisfies (ii) or (iii) from Proposition~\ref{dim4}. Then $V$ lifts to a complex module $V_\CC$. 
As $d\geq 3$, there is a p.p.d. $r$ for $(3,2d-1)$, see Lemma~\ref{LPPD}. By Lemma~\ref{dim1}(i), $r$ divides $\dim V_\CC$. On the other hand, $r{\not|}\,|H|$, so $V_\CC$ and $V$ are both reducible over $H$. 
\end{proof}

\subsection{On irreducible characters of $\mathbf{Sp_{2n}(q)}$} In this subsection we use some facts from Deligne-Lusztig theory which can be found e.g. in \cite{DM}. 

\begin{Lemma}\label{sp1}
Let $p$ be odd, $n \geq 3$, $(n,q) \neq (3,3)$, $S = Sp_{2n}(q)$, $\chi \in \Irr(S)$ have degree coprime to $p$, and $1 < \chi(1) \leq (q^{2n}-1)/(q-1)$. 
Then $\chi$ is one of the following characters:
\begin{enumerate}
\item[{\rm (i)}] four Weil characters of degree $(q^{n}-\ep)/2$, two for each $\ep = \pm 1$;

\item[{\rm (ii)}] four characters of degree $(q^{2n}-1)/(2(q-\ep))$, two for each $\ep = \pm 1$; 

\item[{\rm (iii)}] $(q-3)/2$ characters of degree $(q^{2n}-1)/(q-1)$;

\item[{\rm (iv)}] $(q-1)/2$ characters of degree $(q^{2n}-1)/(q+1)$.
\end{enumerate}
\end{Lemma}

\begin{proof}
By Lusztig's classification irreducible characters 
of $S$ corresponds to pairs $((s),\psi)$, where $(s)$ is a semisimple conjugacy class in 
$S^{*} := SO_{2n+1}(q)$ and $\psi$ is a unipotent character of $C := C_{S^{*}}(s)$, with 
$\chi(1) = D\psi(1)$, where $D=(S^{*}:C)_{p'}$. Consider the natural module 
$\MC = \FQ^{2n+1}$ for $S^{*}$.

1) If $\MC = \AC \oplus \BC$ is an orthogonal sum of two proper $C$-invariant subspaces, 
then one of the subspaces $\AC$, $\BC$ 
has dimension $1$ or $2$. Indeed, otherwise  
we may assume $3 \leq \dim \AC = 2k+1 \leq 2n-3$. 
If $5 \leq 2k+1 \leq 2n-5$ then $n \geq 5$, and 
$C \leq S^{*} \cap (O(\AC) \times O(\BC)) = (SO_{2k+1}(q) \times SO^{\pm}_{2(n-k)}(q)) \cdot 2,$
whence 
$D > q^{(n-k)(2k+1)}/8 \geq q^{2n+5}/8 > q^{2n+3}.$
If $\dim \AC = 2n-3$, then 
$\dim \BC = 4$, $C \leq (SO_{2n-3}(q) \times SO^{\pm}_{4}(q)) \cdot 2$, and so 
$$D \geq (q^{2n-1})(q^{2n-2}-1)/2(q^{2}-1)(q^{2} \mp 1) > (q^{2n}-1)/(q-1)$$ as $n \geq 3$ and 
$(n,q) \neq (3,3)$. Finally, if $\dim \AC = 3$, then $\dim \BC = 2n-2$, 
$C \leq (SO_{3}(q) \times SO^{\pm}_{2n-2}(q)) \cdot 2$, and so 
$$D \geq (q^{2n-1})(q^{2n-2}-1)/2(q^{2}-1)(q^{n-1} \mp 1) > (q^{2n}-1)/(q-1)$$ 
again, because   
$n \geq 3$ and $(n,q) \neq (3,3)$.

2) One can show (cf. the proof of \cite[
5.2]{TZ0}) that $\MC$ is an orthogonal sum $\oplus^{t}_{i=1}\MC_{i}$ of 
$C$-invariant nondegenerate subspaces, with summands satisfying one of the following: (a) $\MC_{i}$ has odd dimension and 
$s|_{\MC_{i}} = 1_{\MC_{i}}$, (b) $\MC_{i}$ has even dimension and $s|_{\MC_{i}} = -1_{\MC_{i}}$, (c) $\MC_{i}$ has even dimension $2kd$, $s \pm 1$ is nondegenerate on $\MC_{i}$, and 
$C_{O(\MC_{i})}(s{|}_{\MC_i}) = GL^{\pm}_{k}(q^{d}) < SO(\MC_{i})$ (where $GL^{\ep}$ stands for $GL$ if $\ep = +$ 
and for $GU$ if $\ep = -$). 
We may assume that $\MC_{1} = \Ker(s-1)$. If $t \geq 3$, we can write $\MC = \AC \oplus \BC$ as
an orthogonal sum of two $C$-invariant subspaces of dimension $\geq 3$, contrary to 1). If $t = 1$, then $1 < \chi(1) = \psi(1)$ is the degree of some nontrivial unipotent character of $S^{*}$, and so it is divisible by $p$, see the proof of \cite[
7.2]{MMT}. It remains to consider the case $t = 2$. By 1) there are two possible possibilities.

Case 1: $\dim \MC_{1} = 1$. If $s|_{\MC_{2}} \neq -1_{\MC_{2}}$, then $C = GL^{\pm}_{k}(q^{d})$ with 
$kd = n$, and so $D > (q^{2n}-1)/(q-1)$. Hence $s|_{\MC_{2}} = -1_{\MC_{2}}$, $C \cong O^{\ep}_{2n}(q)$.
As $\psi$ is a 
unipotent character of $O^{\ep}_{2n}(q)$ of degree coprime to $p$, we have $\psi(1)=1$ as above. 
In fact, 
$C$ has two unipotent characters of degree $1$ for each $\ep = \pm$. This leads to 
four characters of degrees as in (i), and these must be Weil characters by \cite[
5.2]{TZ0}.

Case 2: $\dim \MC_{1} = 2n-1$. Suppose $s|_{\MC_{2}} \neq -1_{\MC_{2}}$. Then 
$C = SO_{2n-1}(q) \times GL^{\ep}_{1}(q)$ with $\ep = \pm$. As above, $\psi(1) = 1$, hence $\psi$ is
trivial, and $\chi(1) = (q^{2n}-1)/(q-\ep)$. Moreover, $s|_{\MC_{2}} \in GL^{\ep}_{1}(q)$ has order $> 2$ 
and $s$ and $s^{-1}$ are conjugate in $S^{*}$. Thus we arrive at (iii) when $\ep = +$ and at (iv) 
when $\ep = -$. Finally, assume $s|_{\MC_{2}} = -1_{\MC_{2}}$. Then 
$C = (SO_{2n-1}(q) \times SO^{\ep}_{2}(q)) \cdot 2$. 
Again $\psi(1) = 1$, and 
$C$ has two unipotent characters of degree $1$ for each $\ep = \pm$. This leads to the conclusion (ii). For a future reference, we fix such a decomposition 
$V = \MC_{1} \oplus \MC_{2}$, where $\MC_{2}$ is a $2$-dimensional orthogonal
subspace of type $+$, and denote by $T$ the subgroup
$\{g \in S^* \mid g|_{\MC_{1}} = 1_{\MC_{1}}, \det(g|_{\MC_{2}}) = 1\}$. Notice
that $T \cong \FQ^{\times}$. 
\end{proof}

\begin{Lemma}\label{bm}
Let $\GC$ be a connected reductive algebraic group in characteristic $p \neq \ell$ and $F$ be a Frobenius map 
on $\GC$. Let the pair $(\GC^{*},F^{*})$ be dual to $(\GC,F)$, and set $G := \GC^{F}$ and 
$\GD := (\GC^*)^{F^{*}}$. Assume that $\chi_{i} \in \Irr(G)$ belongs to the 
rational series $\EC(G,(s_{i}))$ corresponding to the $\GD$-conjugacy class of the semisimple
element $s_{i} \in \GD$, and $t_{i}$ is the $\ell'$-part of $s_{i}$, for $i = 1,2$. 
\begin{enumerate}
\item[{\rm (i)}] If $t_{1}$ and $t_{2}$ are not $\GD$-conjugate, then $\chi_{1}$ and 
$\chi_{2}$ belong to different $\ell$-blocks of $G$.

\item[{\rm (ii)}] If $\chi_{1}(1) = (\GD:C_{\GD}(t_{1}))_{p'}$, then $\chi_{1} (\mod \ell)$ is 
irreducible.
\end{enumerate}
\end{Lemma}

\begin{proof}
Set $\EC_{\ell}(G,(t_{i})) = \cup\,\EC(G,(ut_{i}))$, the union over all $\ell$-elements $u\in C_{\GD}
  (t_{i})$. 
By 
\cite{BM}, $\EC_{\ell}(G,(t_{i}))$ 
is a union of $\ell$-blocks, so we get (i), as distinct rational series are disjoint. Next,
by \cite{HM}, the degree of any irreducible Brauer character in $\EC_{\ell}(G,(t_{1}))$ is divisible by $(\GD:C_{\GD}(t_{1}))_{p'}$, giving (ii).  
\end{proof}

\subsection{Restrictions to $\mathbf{Sp_{2n}(q)}$}
\begin{Proposition}\label{sp2}
Let $S = Sp_{2n}(q) < H = CSp_{2n}(q) < G = GL_{2n}(q)$ with $n\geq 2$, and 
$V = L(\si,(1)) \circ L(\tau,(2n-1))$ for $\ell'$-elements $\si\neq\tau$  in $\FQ^{\times}$. Then: 
\begin{enumerate}
\item[{\rm (i)}] If $\tau \neq \pm \si$ then $V\dar_{S}$ is irreducible.

\item[{\rm (ii)}] If $\tau = -\si$ 
then $V\dar_{H}$ is irreducible and 
$V\dar_{S}$ is reducible.
\end{enumerate}
\end{Proposition}

\begin{proof}
By Lemma~\ref{tensor} we may assume that $\tau = 1$. If $q$ is even, the statement has been proved in \cite[
7.10]{GT2}. Let $q$ be odd. 
 
1) The $1$-dimensional $GL_1$-module $L(\si,(1))$ defines a nontrivial
character $\al:\FQ^{\times} \to \FF^{\times}$. Fix $\gamma \in \FQ^{\times}$ with $\al(\gamma) \neq 1$ and a Witt basis $e_{1},f_{1},\ldots ,e_{n},f_{n}$ 
for the natural module $\NC = \FQ^{2n}$ of $S$. Let $A = \FF\cdot v$ be 
the $1$-dimensional module of $P = \Stab_{G}(\langle e_{1} \rangle)$, where   
$hv = \al(\beta_{h})v$ for any $h \in P$ with $h(e_{1}) = \beta_{h}e_{1}$. Then 
$V = A \uar^{G}$. If $X \leq G$ is a subgroup with $G = XP$,
$$\textstyle \End_{X}(V) \cong \bigoplus_{RgR \in R \backslash X/R}
  \Hom_{R \cap R^{g}}(A\dar_{R \cap R^{g}},A^{g}\dar_{R \cap R^{g}}),$$
where $R := X \cap P$. For $X \in \{H,S\}$, we have  
$X = \sqcup^{3}_{i=1}Rg_{i}R$ with
$g_{1} = 1$, $g_{2}(e_{1}) = e_{2}$, and $g_{3}(e_{1}) = f_{1}$. Set 
$$N_{i} := \dim\Hom_{R \cap R^{g_{i}}}(A\dar_{R \cap R^{g_{i}}}, A^{g_{i}}\dar_{R \cap R^{g_{i}}}).$$ 
Then obviously $N_{1} = 1$. Furthermore, 
$N_{2} = 0$ as $R \cap R^{g_{2}}$ contains an element $y$ with $y(e_{1}) = e_{1}$ and
$y(e_{2}) = \gamma e_{2}$. If $X = H$, then $N_{3} = 0$, as $R \cap R^{g_{3}}$ contains an 
element $z$ with $z(e_{1}) = e_{1}$ and $z(f_{1}) = \gamma f_{1}$. If $X = S$, then for any element
$z \in R \cap R^{g_{3}}$ we have $z(e_{1}) = \beta_{z}e_{1}$ and 
$z(f_{1}) = \beta_{z}^{-1}f_{1}$, whence $N_{3} = 0$ if $\al \neq \al^{-1}$ and $N_{3} = 1$ if 
$\al = \al^{-1}$. We have shown that $\dim\End_{X}(V)$ equals $1$ if $X = H$, or if
$X = S$ and $\si \neq -1$, and $2$ if $X = S$ and $\si = -1$. In particular, we are done if 
$\ell = 0$. 

Set $\VC(\de) := L_\CC(\de,(1)) \circ L_\CC(1,(2n-1))$ for $\de \in \FQ^{\times}$. Using 
Mackey's formula as above, one can check that $\Hom_{S}(\VC(\de)\dar_{S},\VC(\de')\dar_{S}) \neq 0$
precisely when $\de' = \de^{\pm 1}$. In particular, $\VC(\de)\dar_{S} \cong \VC(\de^{-1})\dar_{S}$ when 
$\de \neq \pm 1$ since the modules are irreducible.

2) If $(n,q) = (3,3)$, then $\si = -1$ and $\ell \neq 2$. By 1), 
$\VC(\si)$ is an irreducible $H$-module of dimension $364$ splitting into two irreducible 
$S$-modules. Using \cite{Atlas} and \cite{JLPW}, we can  verify (ii). The case 
$(n,q) = (2,3)$ is checked similarly.

3) Now assume $(n,q) \neq (2,3)$, $(3,3)$. Let $\chi_{\de}$ be the character of 
$\VC(\de)$. By 1), $\chi_{\de}\dar_{S}$ is irreducible of degree $(q^{2n}-1)/(q-1)$ if $\de \neq \pm 1$,
and $\chi_{-1}\dar_{S} = \rho_{1} + \rho_{2}$, the sum of two irreducible characters of degree
$(q^{2n}-1)/(2(q-1))$.

If $n \geq 3$, then by Lemma~\ref{sp1},  $\chi_{\de}\dar_{S}$ for $\de \neq \pm 1$  
is the semisimple character $\varphi_{t}$ labeled by an element $t$ of 
a certain subgroup $T$ of the dual 
$S^{*} = SO_{2n+1}(q)$, $t^{2} \neq 1$, and 
$\varphi_{t}(1) = (S^{*}:C_{S^{*}}(t))_{p'}$. This $T$ appears in Case 2 of the 
proof of Lemma~\ref{sp1}, and is identified with $\FQ^{\times}$.
Moreover, $\rho_{1}$ and $\rho_{2}$ 
are contained in the rational series $\EC(S,(j))$ for an involution $j \in S^{*}$
with $(S^{*}:C_{S^{*}}(j))_{p'} = \rho_{1}(1) = \rho_{2}(1)$. The same statements  hold 
for $n = 2$. Indeed, $\chi_{\de}\dar_{S}$ with $\de \neq \pm 1$ is some $\xi_{3}(k)$, and $\rho_{1}$ and 
$\rho_{2}$ are 
$\Phi_{5}$ and $\Phi_{6}$ in the notation of \cite{Sr}; their 
Jordan decomposition is given in \cite{Wh}. 

Write $\FQ^{\times} = A \sqcup B$, where $A$ is the set of $t$ whose $\ell'$-part is $1$ or $-1$, and 
let $C := O_{\ell}(\FQ^{\times})$. Set $\kappa = 1$ if $\ell = 2$ and 
$\kappa = 2$ if $\ell \neq 2$. Note that $|C|$ divides $|A|$ (in fact $|A| = \kappa |C|$) and 
$|B|$. 
Consider $t, t' \in B$. Since $t$ and the $\ell'$-part of $t$ have the same centralizer in $S^{*}$, 
by Lemma~\ref{bm} we see that $\varphi_{t} (\mod \ell)$ is irreducible;
moreover, $\varphi_{t}$ and $\varphi_{t'}$ belong to different $\ell$-blocks if 
$t^{\pm 1}t' \notin C$ (indeed, in this case the $\ell'$-parts of $t$ and $t'$ are not 
$S^{*}$-conjugate). We noted above that $\chi_{\de}\dar_{S} \neq \chi_{\de'}\dar_{S}$ 
whenever $\de' \neq \de^{\pm 1}$. So, when we vary $\de \in \FQ^{\times}\setminus\{\pm 1\}$, the restrictions
$\chi_{\de}\dar_{S}$ yield all  irreducible characters of degree $(q^{2n}-1)/(q-1)$ of $S$,
which are the $\varphi_{t}$ with $\pm 1 \neq t \in \FQ^{\times}$. 

Moreover, when we vary $\de \in \FF_q^\times$, the restrictions
of $\chi_{\de} (\mod \ell)$ to $S$ yield at least $|B|/\kappa|C|$ distinct irreducible Brauer characters 
of $S$ of degree $(q^{2n}-1)/(q-1)$. On the other hand, by \cite{J}, $\chi_{\de} (\mod \ell) = \chi_{\de'} (\mod \ell)$ if 
$\de \equiv \de' (\mod C)$. In particular, if $\de \in A$ then the restriction of $\chi_{\de}(\mod \ell)$ to 
$S$ equals $\chi_{\pm 1}(\mod \ell) \dar_{S}$ which is reducible as $\chi_{\pm 1}\dar_{S}$ is reducible.
Thus when we vary $\de \in \FQ^{\times}$, the restrictions of $\chi_{\de} (\mod \ell)$ to $S$ yield at most 
$(|\FQ^{\times}|-|A|)/\kappa |C| = |B|/\kappa|C|$ irreducible Brauer characters of $S$ of degree $(q^{2n}-1)/(q-1)$ as composition factors, and this 
restriction can be irreducible only when $\de \in B$. Consequently, $\chi_{\de}(\mod \ell)\dar_{S}$ is 
irreducible precisely when $\de \in B$. In particular, $V\dar_{S}$ is irreducible if $\si \neq \pm 1$.

4) Finally, let $\si = -1$ (so $\ell \neq 2$). Recall that $\rho_{1}$ and 
$\rho_{2}$ are irreducible characters in the rational series $\EC(S,(j))$ for some 
involution $j \in S^{*}$ and $(S^{*}:C_{S^{*}}(j))_{p'} = \rho_{1}(1) = \rho_{2}(1)$.  
By Lemma~\ref{bm}, $\rho_{i} (\mod \ell)$ is irreducible for $i = 1,2$, since $\ell \neq 2 = |j|$. 
Set $P_{m} := \Stab_{H}(\langle e_{1},\ldots,e_{m}\rangle)$. Then
$U := O_{p}(P_{m})$ is elementary abelian, 
and $\chi_{-1}\dar_{U} = \chi_{1}\dar_{U}$. But $\chi_{1}{\dar}_S = 1_{P \cap S} \uar S$ is 
the permutation
character of 
$S$ on $1$-spaces of $\NC$. So, by \cite{LST}, 
$\chi_{1}\dar_{U}$ affords two $P_{m} \cap S$-orbits of nontrivial linear characters of $U$ of 
length $(q^{n}-1)/2$, fused by $P_m$, each with multiplicity $1$ (and some other orbits). As $\chi_{-1}\dar_{H}$ is irreducible, one of them is 
afforded by $\rho_{1}\dar_{U}$ and the other by $\rho_{2}\dar_{U}$. Thus the two Brauer 
characters $\rho_{i} (\mod \ell)$, $i = 1,2$, are irreducible, $H$-conjugate, but distinct. So 
$V\dar_{H}$ (affording 
$(\rho_{1}+\rho_{2}) (\mod \ell)$ over $S$) is
irreducible.          
\end{proof}

\begin{Theorem}\label{main-sp}
Let $n\geq 2$, $G = GL_{2n}(q)$, $H = CSp_{2n}(q)$, $S = Sp_{2n}(q)$, $V \in \IBr(G)$, $\dim V > 1$. Then $V{\dar}_H$ is irreducible if and only if
$V = L(\si,(1)) \circ L(\tau,(2n-1))$ for some $\ell'$-elements $\si, \tau \in \FF_{q}^{\times}$
with $\tau \neq \si$, in which case $\dim V = (q^{2n}-1)/(q-1)$. Moreover, such a module $V$ 
is irreducible over $S$ if and only if $\tau \neq \pm \si$.
\end{Theorem}

\begin{proof}
1) The `if' part follows from Proposition~\ref{sp2}. Moreover, the theorem is checked directly for $(n,q) = (2,2)$ (see \cite{Atlas} and \cite{JLPW}). So let $(n,q) \neq (2,2)$.
Pick a Witt basis $e_{1},f_{1}, \ldots ,e_{n},f_{n}$ 
for the natural module $\NC = \FQ^{2n}$ of $S$. If $J$ is the Gram matrix of the symplectic form in 
this basis, then for any $h \in H$, $\tn hJh = \al(h)J$ for some $\al(h) \in \FQ^{\times}$. 
Define the involutory automorphism $\be$ of 
$G$ via $\be(g) = J \tn g^{-1} J^{-1}$. 
Then $\be(h) = h (\mod Z(G))$ for all $h \in H$. 

2) First we assume that $V$ is Harish-Chandra indecomposable, i.e. 
$V = L(\si,\la)$ for some $\si$ of degree $d|2n$ and $\la \vdash 2n/d$. Consider three cases. 

Case 1: $d = 1$. Using Lemma~\ref{tensor} we may assume $\si = 1$,  
$\la \neq (2n)$. Then $Z(G)$ acts trivially on $V$, so $V$ can be considered as a module over $G/Z(G)$. 
Let $\ga$ be  the inverse-transpose automorphism of $G$. By \cite[(7.33)]{J}, $V$ is $\ga$-invariant. 
Hence $V$ is $\be$-invariant, too. So
$V$ extends to a module $\tilde{V}$ over $\tilde{G} := \langle G/Z(G),\be \rangle$. Now 
$HZ(G)/Z(G) \leq C_{\tilde{G}}(\be)$, as $\be(h) = h (\mod Z(G))$ for all $h \in H$. Consequently, if $V\dar_{H}$ is irreducible, then $\be$ acts  scalarly
on $\tilde{V}$, and so $\be$ centralizes $G/Z(G)$, giving a contradiction.  

Case 2: $d = 2$. By Lemma~\ref{LPPD}, there is a p.p.d. $r$ for $(q,2n-1)$. Note that $r{\not{|}}\,|H|$. By Lemma~\ref{dim1}(i), if $\la = (n)$, then $V$ lifts to a complex module $\VC$ and $r|\dim \VC$, whence $\VC\dar_{H}$ is reducible and so is $V\dar_{H}$. Thus $\la \neq (n)$. 
By \cite[5.5d]{BDK}, $\dim L(\si,\la) = N  \dim L^{(2)}(1,\la)$, where 
$N=\prod^{n}_{i=1}(q^{2i-1}-1)$ and $L^{(2)}(1,\la)$ is an irreducible $GL_n(q^2)$-module corresponding to the symbol $(1,\la)$. 
As $\la \neq (n)$, we have $\dim L^{(2)}(1,\la) \geq q^{2n-2}-1$ by \cite{GT1,BK}. Hence
$\dim V = N  \dim L^{(2)}(1,\la) > (q-1)q^{n^{2}+2n-3}/2,$ 
and so 
$\dim V > \ml(H)$ if $n \geq 3$. If $n = 2$, then $q \geq 3$, and 
$\dim V = \dim L(\si,(1^{2})) \geq (q-1)(q^2-1)(q^{3}-1) > \ml(H).$ We conclude that $V\dar_{H}$ is 
reducible in the Case 2.

Case 3: $d \geq 3$. Note that 
$\ml(H) = \ml(S) < \sqrt{|S|} < q^{n^{2}+n/2}$
if $q = 2$, 
$\ml(H) \leq (q-1)\ml(S) \leq (q-1)(\prod^{n}_{i=1}(q^{2i}-1))/(q-1)^{n} < q^{n^{2}+n-1}$
if $n,q \geq 3$, and $\ml(CSp_{4}(q)) < (q+1)(q^{4}-1)$, cf. \cite{S2}. On the other hand,
$\dim V > (q-1)q^{4n^{2}/3-1}/2$ by Lemma~\ref{dim1}(iii). Hence $\dim V > \ml(H)$ if
$n,q \geq 3$, or if $q = 2$ and $n \geq 4$. If $(n,q) = (3,2)$, then 
$\dim V \geq 1395 > 720 = \ml(H)$. If $n = 2$ then $q \geq 3$ and $d = 4$, hence
$\dim V = (q-1)(q^{2}-1)(q^{3}-1) > \ml(H).$ 
Thus in all cases $V\dar_{H}$ is reducible.

3) Now we may assume that $V$ is induced from the stabilizer $P$ in $G$ of a $k$-space $\MC\subset\NC$, $1 \leq k \leq n$. If 
$k \geq 2$, then $H$ has at least two orbits on $k$-spaces of
$\NC$, one containing a totally isotropic subspace 
and
another containing a non-totally isotropic subspace. 
Then $V|\dar_{H}$ is reducible by Mackey's Theorem. So we may assume
$k = 1$, $\MC = \MC_{1} := \langle e_{1} \rangle$, and 
$V = L(\si,(1)) \circ L(\tau,\la)$, where $\si\neq\tau$ are $\ell'$-elements, $\deg(\si) = 1$, 
$\deg(\tau) = d|(2n-1)$ and $\la \vdash (2n-1)/d$.

4) 
Set 
$\MC_{2} := \langle e_{2}, \ldots ,e_{n},f_{1},f_{2}, \ldots ,f_{n}\rangle$. 
Consider the module $W = L(\si,(1)) \otimes L(\tau,\la)$ over the Levi subgroup 
$L := GL(\MC_{1}) \times GL(\MC_{2}) = GL_{1} \times GL_{2n-1}$
of $P = UL$ with $U = O_{p}(P)$. 
Inflating $W$ to $P$, we have $V = W \uar_P^{G}$. 
If $\dim W=1$, then $\dim L(\tau,\la) = 1$, whence   
$d = 1$, $\la = (2n-1)$. So it suffices to show that $V\dar_{H}$ is irreducible implies $\dim W = 1$. 
Assume for a contradiction that $V\dar_{H}$ is irreducible, but $\dim W>1$. 

We can write $L = Z(G) \times L_{2}$, where 
$L_{2} := GL(\MC_{2})$. Clearly, $W\dar_{L_{2}}$ is irreducible. Since 
$V\dar_{H} = (W\dar_{H \cap P}) \uar H$, $W\dar_{H \cap P}$ is irreducible. Set  
$Q_{1} := H \cap P$, $Q_{2} := \Stab_{L_{2}}(\langle e_{2}, \ldots ,e_{n},f_{2}, \ldots ,f_{n}\rangle)$. Then, for $i=1,2$, $U_{i} = O_{p}(Q_{i})$ has order $q^{2n-i}$. Moreover, 
$Q_{1} < Z(G)UQ_{2}$ and $U \lhd UU_{1} = UU_{2}$.

Case 1: $W^{U_{2}} \neq 0$. As $U$ acts trivially on $W$ and 
$UU_{1} = UU_{2}$, we see that $W^{U_{1}} \neq 0$, but $U_{1}$ does not act trivially on $W$ 
(otherwise $U_{2} < L_{2}$ acts trivially on $W$ and so $\dim W = 1$). Thus 
$W\dar_{H \cap P} = C_{W}(U_{1}) \oplus [W,U_{1}]$ is a direct sum of two nonzero submodules and so it
is reducible.

Case 2: $W^{U_{2}} = 0$. Now $W\dar_{L_{2}} = L(\tau,\la)$ is an 
irreducible $L_{2}$-module, and $U_{2}$ is the unipotent radical of the parabolic $Q_{2}<L_{2}$. Applying Theorem~\ref{TQ} to $L(\tau,\la)$, we see that either $d > 1$, or 
$d = 1$ and $e(1){|}\la'$. In the former case $d \geq 3$ and in the latter case $e(1) \geq 3$, as
they divide $2n-1$. In either case,  
$\dim W = \dim L(\tau,\la) > (q-1)q^{(2n-1)^{2}/3-1}/2$ by Lemma~\ref{dim1}(iii),(iv). So
$$\dim V = \dim L(\tau,\la) \cdot (q^{2n}-1)/(q-1) >(q^{2n}-1)q^{(2n-1)^{2}/3-1}/2 > \ml(H)$$ 
if $n \geq 3$. If $n = 2$, then $q \geq 3$, $\dim L(\tau,\la) \geq (q-1)(q^{2}-1)$ and 
$\dim V \geq (q^{2}-1)(q^{4}-1) > \ml(H)$. Thus $V\dar_{H}$ is reducible.
\end{proof}

\subsection{Restrictions to $\mathbf{G_{2}(q)}$}

\begin{Proposition}\label{g2}
Let $q = 2^{f}$, $G = GL_{6}(q)$, $G_{2}(q)' \lhd H \leq G$, and 
$V \in \IBr(G)$ with $\dim V > 1$. Then $V{\dar}_H$ is irreducible if and only if
$V = L(\si,(1)) \circ L(\tau,(5))$ for some $\ell'$-elements 
$\si\neq\tau \in \FF_{q}^{\times}$. 
\end{Proposition}

\begin{proof}
The case $q = 2$ is a direct check 
\cite{Atlas,JLPW}. Let $q > 2$. 
The 
group $G_{2}(q)$ acts irreducibly on the natural module $\NC$. But
the (nontrivial) field automorphisms of $G_{2}(q)$ do not stabilize this action, so
$H = Z(H) \times G_{2}(q)$, and we may assume that $H = G_{2}(q)$. 

As $\dim V \leq \sqrt{|G_2(q)|}$, 
\cite{GT1} or \cite[3.4]{BK} implies that $V = L(\tau,(5,1))$, or $V = L(\si,(1)) \circ L(\tau,(5))$, for 
some $\ell'$-elements $\si, \tau \in \FF_{q}^{\times}$. Using Lemma~\ref{tensor} we 
may assume $\tau = 1$. By \cite{GT2}, $L(1,(5,1))$ is reducible 
over $Sp_{6}(q) > H$, so we may assume $V = L(\si,(1)) \circ L(1,(5))$. In particular, 
$\dim V = (q^{6}-1)/(q-1)$.

We consider the restriction to $H$ of the characters $\chi_{\si}$ of the 
$G$-module $\LC(\si,(1)) \circ \LC(1,(5))$. 
By \cite{EY}, all 
irreducible characters of degree $(q^{6}-1)/(q-1)$ 
are induced from one of
the two maximal parabolic subgroups $P,Q$ in the notation of \cite{EY}, where 
$Q$ is the stabilizer in $H$ of a $1$-space in $\NC$. 
Now   
$H$ has $(q-2)/2$ irreducible characters $\chi_{3}(k)$, $1 \leq k \leq (q-2)/2$, of degree
$(q^{6}-1)/(q-1)$, which are induced from $Q$. 
Inspecting 
the values of these 
characters at the classes $C_{21}(i)$ (in the notation of
\cite{EY}), 
we see that $\chi_{\si}\dar_{H} = \chi_{\si^{-1}}\dar_{H}$, and the 
restrictions of $\chi_{\si}$ to $H$, with $1 \neq \si \in \FQ^{\times}$, yield 
the aforementioned $(q-2)/2$ irreducible characters $\chi_{3}(k)$. We are done in the case $\ell = 0$. Now, let $\ell>0$. 

By \cite{H}, the characters $\chi_{3}(k)$ are 
the semisimple characters corresponding to the conjugacy classes in $H^{*} \cong H$ of
certain elements $h_{1a}(i)\in \FQ^{\times}$. Using Lemma~\ref{bm} (or information 
on the 
decomposition matrix \cite{H}), one sees that 
among the $\chi_{3}(k) (\mod \ell)$ there are at least $m := ((q-1)_{\ell'}-1)/2$ distinct
irreducible Brauer characters of $H$. 
On the other hand, by \cite{J}, 
$\chi_{\si} (\mod \ell) = \chi_{\tau} (\mod \ell)$ if $\si_{\ell'}=\tau_{\ell'}$, and, as mentioned above, $\chi_{\si}\dar_{H} = \chi_{\si^{-1}}\dar_{H}$. 
In particular, if $\si$ is an $\ell$-element then 
$\chi_{\si}(\mod \ell){\dar}_H=\chi_{1}(\mod \ell) \dar_{H}$, which is reducible,  
as $\chi_{1}$ is reducible. 
So, if we vary $\si \in \FQ^{\times}$, the 
restrictions $\chi_{\si} (\mod \ell)\dar_H$ yield at most 
$m$ irreducible Brauer characters of $H$, and this restriction can be irreducible only 
when $\si$ is not an $\ell$-element. Hence $\chi_{\si}(\mod \ell)\dar_{H}$ is 
irreducible if and only if $\si$ is not an $\ell$-element. In particular, $V\dar_{H}$ 
is irreducible if $\si \neq 1$.
\end{proof}

\section{Proofs of main theorems and further results}

\subsection{Proof of Theorem~\ref{main1}.} 

If $H$ is reducible on the natural module $\NC$, then $H$ is contained in a parabolic subgroup of $G$. Applying Theorem~\ref{para} and Propositions~\ref{P11}, \ref{P12} we  arrive at the cases 
in Theorem~\ref{main1}(i). 
From now on we assume that $H$ is irreducible on 
$\NC$. 
By Lemma~\ref{semi1} and Propositions \ref{semi2},  
\ref{semi3}, we may also assume that 
$H \not\leq \Gamma L_{d}(q^{s})$ for any $s = n/d$ with $s > 1$. 

Let $W$ be as in Lemma~\ref{L2M}. By the lemma, either $V\dar_H$ is reducible or 
the condition (b) of Proposition~\ref{large} holds. In view of the proposition and using the above assumptions on
$H$, we are left with only one possibility:  
$q$ is a square, $n$ is even, and $H \leq CU_{n}(\sqrt{q})$. But 
$|CU_{n}(\sqrt{q})|$ cannot be divisible by any p.p.d. for 
$(p,(n-1)f/2)$. 
This contradicts Lemma~\ref{L2M}. From now on we assume that $W$ is not as in Lemma~\ref{L2M}.

Assume that $H$ is intransitive 
on $1$-spaces of $\NC$. The irreducibility of $H$ on $V$ implies by Corollary~\ref{CRed} that $\dim\End_{P_1 \cap G}(V) = 1$.  
Applying Corollary \ref{dim3} to $V$, we conclude 
that $\dim V> q^{(n^{2}+5)/4}$. So 
$|H/Z(H)| \geq (\dim V)^{2} > q^{(n^{2}+5)/2}$. 
Now Proposition~\ref{large} implies 
that $n$ is even and $Sp_n(q)\leq H \leq CSp_{n}(q)$. So $H$ is transitive on $1$-spaces of $\NC$, a contradiction.

Finally, let $H$ be transitive on $1$-spaces of $\NC$.
We can apply Proposition~\ref{step1}. Note that 
$SL_{2}(13)$ is reducible on 
$V$ (as $\dim V \geq 362$ \cite{GT1,BK}). 
Thus $H \rhd Sp_{n}(q)$ or $H \rhd G_{2}(q)'$. If $W\dar_{G}$ is irreducible, then  by Theorem~\ref{main-sp} and Proposition~\ref{g2}, $W{\dar}_H$ is irreducible if and only if the 
possibility (ii) in Theorem~\ref{main1} occurs. If $W\dar_{G}$ is reducible, then by  Proposition~\ref{dim4},
$\dim W > (q-1)q^{(n^{2}+n)/4}$, whence $\dim V > q^{(n^{2}+n)/4}$ and so 
$|H/Z(H)| > q^{(n^{2}+n)/2}$. This  excludes both possibilities for $H$.

\subsection{Aschbacher's program and Problem \ref{restr} for groups of type $A$}
Consider the situation $H < G < \Gamma$ considered right before Problem 
\ref{restr}, with both $H, G \in \SCL$. We will explain how one can use Theorem 
\ref{main1} to find all possibilities for such $H$, if $G$ is assumed to be an 
almost quasi-simple Lie-type group of type $A$, under the mild assumption that 
$H \notin \CL_{2} \cup \CL_{4}$.

\begin{Corollary}\label{typeA}
Let $G$ be a finite almost quasi-simple group, with $PSL_{n}(q)$, $n \geq 5$, as its unique 
non-abelian composition factor, and let $H < G$ be an almost quasi-simple subgroup. Let 
$V \in \IBr(G)$ be such that $\dim(V) > 1$, ${\ell{\not{|}}}q$, and $V\dar_{H}$ be 
irreducible, primitive, and tensor-indecomposable. Then one of the following holds for 
$L := H^{(\infty)}$ and $S := G^{(\infty)}$. 
\begin{enumerate}
\item[{\rm (i)}] $L = S$ is a quotient of $SL_{n}(q)$ by a central subgroup.

\item[{\rm (ii)}] $n$ is even, $V$ is a Weil representation of degree $(q^{n}-1)/(q-1)$
(when viewed as an $SL_{n}(q)$-representation), and either
$L/Z(L) \cong PSp_{n}(q)$, or $2|q$ and $L \cong G_{2}(q)'$. 

\item[{\rm (iii)}] Viewing $V$ as the module over the full cover $\hat S=SL_n$ of $S$, $V$ is an irreducible constituent of $L(\si,(n/2))\dar_{SL_n}$. Moreover, denote by $M$ the standard subgroup $GL_{n-1}(q)\times GL_1(q)<GL_n(q)$ and by $\hat L$ the inverse image of $L$ in $\hat S$. Then one of the following holds:
\begin{enumerate}
\item[{\rm (a)}] $q=3$, $\si^2=-1$ if $\ell\neq 2$, and $[M,M]\leq \hat L\leq M$;
\item[{\rm (b)}] $q=2$, 
$\si\neq1=\si^3$ and $\hat L=M$. 
\end{enumerate}
\end{enumerate}
\end{Corollary}
\vspace{-3 mm}
\begin{proof}
The assumptions on $G,H$ imply that $S$ is a quotient of $SL_{n}$ (by a central
subgroup), and $S \geq L$. Assume that (i) does not hold. Now $L \lhd H$ and $V\dar_{H}$ is 
irreducible and primitive, whence $V\dar_{L}$ is homogeneous. But $V\dar_{H}$ is also
tensor-indecomposable, so $V\dar_{L}$ is  irreducible, and so is $V\dar_{S}$. It remains to apply Theorem \ref{main1} to the triple 
$(\hat S,V\dar_{\hat S},\hat L)$.
\end{proof}
 
 \vspace{-.5 mm}
The analogue of Corollary~\ref{typeA} for $2 \leq n \leq 4$ has extra examples, including the ones arising because of the exceptional 
Schur multiplier of $PSL_{n}(q)$ for some  $(n,q)$. Here are three examples 
of this kind: $(G,H,\dim V,\ell) = (3\cdot PSL_{2}(9),\AAA_{5},3,\ell \neq 2)$, 
$(6 \cdot PSL_{2}(9),SL_{2}(5),6,\ell \neq 2,5)$, $(2\cdot SL_{4}(2),2\AAA_{7},16,7)$. 
 
 \vspace{-.5 mm}
\subsection{Small rank}\label{SSR} We  now deal with the remaining cases  $2\leq n\leq4$. 

\vspace{-1 mm}
\begin{Theorem}\label{main2}
Let $2\leq n\leq 4$, with $q\not\in\{3,5,7, 9\}$ if $n=2$, $SL_n(q) \leq G \leq GL_n(q)$, $H<G$ be a proper subgroup not containing $SL_n(q)$, and $V$ be an  
irreducible $\FF G$-module of dimension greater than $1$. Let $W$ be an irreducible $\FF GL_n(q)$-module such that $V$ is an irreducible constituent of $W\dar_G$. Then  $V\dar_{H}$ is irreducible if and only if 
one of the following holds: 
\begin{enumerate}
\item[{\rm (i)}] The same as Theorem~\ref{main1}(i) for $n=3,4$. 
\item[{\rm (ii)}] The same as Theorem~\ref{main1}(ii) for $n=4$.
\item[{\rm (iii)}] The same as Lemma~\ref{Ln=2} for $n=2$. 
\item[{\rm (iv)}] $q=2$ or $3$, $G=SL_4(q)$, 
$\dim V=q^3-1$, 
 $HZ(G)=UMZ(G)$, where $U=O_p(P)$ is the unipotent radical of the parabolic $P=UL$ which is the stabilizer of a $1$-space or a $3$-space, and the group $GL_1(q^3)\lhd M\leq \Ga L_1(q^3)$ is naturally embedded into $L$. 
\item[{\rm (v)}] $(G,H,V,\ell)$ from Table I.
\end{enumerate}
\end{Theorem}

\vspace{-3.5mm}
\begin{figure}[ht]
\begin{center}
{\normalsize {\sc Table I.} Small rank examples of irreducible restrictions}
\vskip2pt
\begin{tabular}{|c|c|c|} \hline \skipb
    $G$ & $H$ & 
  $V,\ \ell$ 
  \\ \hline \hline \skipa
   $GL_{4}(2)$  & $ \AAA_{7}$ & 
  $
  \begin{array}{l}
   \dim V = 13, 
   \ \ell = 3,5,\\ 
                     \dim V = 14, 
                     \ \ell \neq 3,5. 
                     \\ 
  \dim V = 21,~\ell \neq 3,5
   \end{array}
  $ \\ \hline \skipa
  $SL_{3}(4)$ & $ 3\AAA_{6}$ & ~$\dim V = 15,~\ell \neq 3$\\
    \hline \skipa
  $GL_{3}(2)$  & $C_{7}:C_{3}$ & ~$\dim V  = 3,~\ell \neq 7$ \\ \hline \skipa
  $SL_{2}(11)\cdot Z(G)$ & $ 2\AAA_{5} \cdot Z(H)$ & 
    $
    \begin{array}{ll}\dim V  = 5, & \ell \neq 2,3,\\ 
                \dim V  = 6, & \ell \neq 2,5. \end{array} 
                $ \\ \hline 
\end{tabular}
\end{center}
\end{figure}

\vspace{-4 mm}
\begin{Remark}
{\rm 
In Table I, 
$V$ can be 
any module with the indicated dimension, unless $\dim V=21$. If $\dim V=21$, $V$ should 
not be $\SSS_8$-invariant. For this and the right embeddings of $H$ one 
should consult  \cite{Atlas,JLPW}.
}
\end{Remark}
\vspace{-3.5 mm}
\begin{proof}
Set $Z = Z(GL_n)$ and replace $H$ by $HZ$ and $G$ by $GZ$.
If $H$ is reducible on 
$\NC$, then $H$ is contained in a parabolic subgroup of $G$. Applying Theorem~\ref{para}, Propositions~\ref{P11}, \ref{P12} and Lemmas~\ref{Ln=2}, \ref{Ln=3}, we arrive at the cases described in Theorem~\ref{main1}(i). 
From now on we assume that $H$ is irreducible on 
$\NC$. 
By Lemma~\ref{semi1} and Propositions \ref{semi2},  
\ref{semi3}, we may also assume that if $n\neq 2$, then 
$H \not\leq \Gamma L_{d}(q^{s})$ for any $s = n/d$ with $s > 1$. 

1) Assume that $n = 3,4$ and $H$ is transitive on the 
$1$-spaces of $\NC$. By Proposition~\ref{step1}, we need to consider the following  
possibilities. 

(a) $G = GL_{4}(2)$ and $H = \AAA_{7}$. Then $\dim V\leq \mc(H) = 35$. Using
\cite{Atlas} and \cite{JLPW} (or \cite{KS2}) we arrive at the first line in Table I.  

(b) $G \leq GL_{4}(3)$ and $E = 2^{1+4} \lhd H \leq N_{G}(E) = E \cdot \SSS_{5}$. 
Then $\dim V \leq \mc(H) \leq 60$. Moreover, the transitivity of $H$ on $1$-spaces of $\NC$
implies that $5||H|$. Thus $\dim V$ is divisible by $4$ if it is faithful on $Z(E)$ and 
by $5$ otherwise. Inspecting \cite{Atlas} and \cite{JLPW}, we see that $\dim V = 40$ and so
$V$ is faithful on $Z(E)$. But the last condition implies $\dim V \leq 24$, a contradiction.

(c) $G = GL_{4}(3)$ and $H \rhd SL_{2}(5) < SL_{2}(9) < SL_4(3)$. Using \cite{Atlas} and 
the transitivity of $H$ on $1$-spaces of $\NC$, one can show that 
$H \leq 2 \cdot (\SSS_{6} \times 2)$ and so $\mc(H) \leq 20$, whereas $\dim V \geq \dl(S) = 26$, 
a contradiction.   

(d) $n = 4$ and $Sp_{4}(q) \lhd H \leq CSp_{4}(q)$. If $V = W\dar_{K}$, then we 
are done by Theorem~\ref{main-sp}. Otherwise $q \neq 2,4$ and, by Proposition~\ref{dim4},
either $W$ is one of the two modules considered in Lemma~\ref{L2M}, or $\dim V \geq (q^{2}-1)(q^{3}-1)$.
In the former case, arguing as in the proof of that lemma, we see that $|H|$ is divisible by a p.p.d. for $(q,3)$, a contradiction. In the latter case, observe that 
$\mc(H) \leq 2\mc(Sp_{4}(q)) \leq 2(q^{2}-1)(q^{4}-1)/(q-1)^{2}$ if $q \geq 5$ and 
$\mc(H) \leq \mc(CSp_{4}(3)) = 120$. Hence $\mc(H) < \dim V$, a contradiction. 

2) Let $n=4$ and $H$ be intransitive on the $1$-spaces of 
$\NC$. 

Case 1: $q \geq 4$. By  Corollaries \ref{CRed} and \ref{dim3}, 
$\dim W \geq (q-1)(q^{3}-1)$, so $\dim V \geq (q-1)(q^{3}-1)/\gcd(4,q-1)$. 
Now apply Aschbacher's Theorem~to  $H/Z<PGL_4(q)$ and check that in all cases $\mc(H) < \dim V$.

Case 2: $q = 3$. Note that either $\dim V \geq 38$, or $\dim V = 26$ and 
$V$ lifts to a complex module. On the other hand, since $H$ is irreducible on $\NC$ and does not
contain $SL_4$, $13{\not{|}}\,|H|$. So $\mc(H) \geq 38$. If $H$ is solvable, 
by the Fong-Swan theorem, $V\dar_{H}$ lifts to an irreducible complex module and 
$\dim V$ divides $|H|$. These conditions exclude all remaining possibilities
for $H$.

Case 3: $q = 2$. Note that either $\dim V \geq 13$, or $\dim V = 7$ and 
$V$ lifts to a complex module. So either $\mc(H) \geq 13$ or $7$ divides $|H|$. Inspecting
the maximal subgroups of $G$, we see that either $H = \SSS_{6}$ or $H \leq \AAA_{7}$. In
the former case, as $\dim V \geq 13$ and $V\dar_{H}$ is irreducible, we must have
$\dim V = 16$, but $G$ does not have an irreducible module of dimension $16$. 
In the latter
case, $V\dar_{\AAA_{7}}$ is irreducible, so $\dim V \geq 13$ by the results of 1). Thus
$\mc(H) \geq 13$ and so $H = \AAA_{7}$, which is transitive, giving a contradiction.

3) Let $n=3$ and $H$ be intransitive on the $1$-spaces of 
$\NC$. As $H$ is irreducible on $\NC$, we have $q \geq 4$. 
If $V=W\dar_G$, apply Corollary~\ref{CRed} and Theorem~\ref{hom22} to deduce that  
$W = L(\si,(1))$ with $\deg(\si)=3$. So 
$\dim W = (q-1)(q^{2}-1)$. 
On the other hand, if $\kappa^{GL_3}_G(W)>1$, then using Theorem~\ref{TMainSL}, one shows that 
$\dim W\geq (q^3-1)$. In either case we have   
$\dim V \geq (q-1)(q^{2}-1)/\gcd(3,q-1)$. 
Now we apply Aschbacher's Theorem~to 
$H/Z<PGL_3$. One check that in all cases $\mc(H) < \dim V$, except
possibly when $q = 4$ and $H/Z \leq \AAA_{6}$. In this case 
$\mc(H) \geq \dim V \geq 15$ implies $H = 3\AAA_{6}$ and $G = SL_{3}(4)$. 
Using \cite{Atlas,JLPW}, we arrive at the second line of Table I. Note that
$3\AAA_{6}$ has three embeddings into $SL_3$ (up to $SL_3$-conjugacy), and there are three pairs of 
irreducible $\FF SL_3$-modules of dimension $15$. For each embedding, 
two of these pairs of 
modules are irreducible over $3\AAA_{6}$.

4) Finally, let $n = 2$. Note that 
$\dim V \geq (q-1)/\gcd(2,q-1)$. In view of our assumptions on $q$, the conditions that $\mc(H) \geq (q-1)/\gcd(2,q-1)$ and $H$ is 
irreducible on $\NC$ leave only one possibility: $q = 11$, $H = Z(H)SL_2(11)$, and 
$G = SL_{2}(5) \cdot Z(G)$. Using \cite{Atlas,JLPW}, we arrive at 
line four of Table I.
\end{proof}

\begin{Remark}
{\rm 
If $n=2$ and $q\in\{3,5,7,9\}$, the irreducible restrictions can be handled using \cite{Atlas,JLPW} or computer. In fact, in addition to the cases similar to Theorem~\ref{main2}(iii), we may have only the following:

(a) If $q=3$ and $\dim V=2$. Such an example arises for instance when $H\geq Q_8$. 

(b) If $q=5$, $\dim V=2$ or $3$, and $G=SL_2(5)Z(G)$. Such examples arise for instance when $H\cong SL_2(3)$.

(c) If $q=7$,  $\dim V=3$ or $4$, and $G=SL_2(7)Z(G)$. Such examples arise for instance when $H\cong 2\SSS_4$.

(d) If $q=9$,  $\dim V=4$ or $5$, and $G=SL_2(9)Z(G)$. Such examples arise for instance when $H\cong SL_2(5)$.
}
\end{Remark}

\subsection{Reducibility of tensor products} As a final application of techniques developed in this paper we complete the analysis of \cite{MT} to prove Theorem~\ref{TTM} below, which also fits nicely into Aschbacher's program as part of the analysis of the family $\mathcal{C}_4$. The following lemma should be compared to Lemma~\ref{hom1};  
(the idea goes back to \cite{BessK} and has been used in \cite{BessKAlt,KT}). 

\begin{Lemma}\label{hom-tensor}
Let $P<G$ be finite groups, $U,V \in \IBr(G)$, and 
the following conditions hold:
\begin{enumerate}
\item[{\rm (i)}] $\dim\End_{P}(U) \geq 2$ and $\dim\End_{P}(V) \geq 2$; 
\item[{\rm (ii)}] The module $W := (1_{P}) \uar^{G}$ is either a direct sum $1_{G} \oplus A$ or 
a uniserial module $(1_{G}|A|1_{G})$ with composition factors $1_{G}$ and $A\not\cong1_G$.
\end{enumerate}
Then $U \otimes V$ is reducible.

\end{Lemma}

\begin{proof}
We have $\dim\Hom_{G}(U^{*} \otimes U,W) = \dim\End_{P}(U)\geq 2$. 
So there exists a $G$-homomorphism  
$\psi:U^{*} \otimes U\to W$ 
whose image is not contained in the unique $G$-submodule $1_G$ of $W$. 
Set $B := A$ (resp. $B := (1_{G}|A)$) if $W = 1_{G} \oplus A$ (resp. if $W = (1_{G}|A|1_{G})$). 
Then $\im\psi\supset B$. 
Similarly, $\dim\Hom_{G}(W,V^{*} \otimes V) = \dim\End_{P}(V)\geq 2$, and there exists a $G$-homomorphism  
$\phi:W\to V^{*} \otimes V$ with $\ker \phi \subsetneq B$. 
Hence $\phi\circ\psi:U^{*} \otimes U\to V^{*} \otimes V$ is a $G$-homomorphism whose image is not 
contained in  $1_{G}$.
But there also exists a $G$-homomorphism 
$\theta:U^{*} \otimes U \to V^{*} \otimes V$ with $\Im(\theta) = 1_{G}$. 
So 
$2 \leq \dim\Hom_{G}(U^{*} \otimes U,V^{*} \otimes V) = \dim\End_{G}(U \otimes V)$, whence $U \otimes V$ is reducible.   
\end{proof}

\begin{Lemma}\label{degree}
Let $2|n=kd \geq 4$ for some integers $k,d$ with 
$d \geq 2$, $G = GL_{n}(3)$, and set $N := \prod^{n}_{i=1}(3^{i}-1)/\prod^{k}_{i=1}(3^{di}-1)$. Then 
$G$ has no irreducible characters of degrees $N (3^{n}-3)/4$ and 
$N (3^{n}-5)/4$. Moreover, if $d\geq 3$ then $G$ has no irreducible characters of degree $N (3^{n}-1)/4$.  
\end{Lemma}

\begin{proof}
Let $\chi \in \Irr(G)$. If $\chi(1) = N (3^{n}-3)/4$, use $\chi(1)||G|$ to deduce 
$(3^{n-1}-1)|4\prod^{n}_{i=1,~i \neq n-1}(3^{i}-1)$. Considering a p.p.d. for $(3,n-1)$, we get a contradiction. 
In the remaining cases $(\chi(1),3) = 1$, so Lusztig's 
classification of 
$\Irr(G)$ implies that 
$\chi(1) = (G:C_{G}(s))_{3'}$ for a 
semisimple 
$s \in G$. Note that 
$C_{G}(s) = \prod^{t}_{i=1}GL_{m_{i}}(3^{a_{i}})$ for some $m_{i},a_{i} \in \NN$ with
$\sum^{t}_{i=1}m_{i}a_{i} = n$.

Assume first that $\chi(1) = N (3^{n}-5)/4$. We have
\begin{equation}\label{d5}
\textstyle  (3^{n}-5)\prod^{t}_{i=1}\prod^{m_{i}}_{j=1}(3^{ja_{i}}-1) = 4\prod^{k}_{i=1}(3^{di}-1).
\end{equation}
As $n \geq 4$, there is a p.p.d. $r_{1}$ for $(3,n)$. But $r_{1}{\not{|}}(3^{n}-5)$,
so (\ref{d5}) implies that $t=1$ and $ma = n$ for $m:= m_{1}$ and $a := a_{1}$. Also,
$k \geq 3$ and $m\geq 2$ by (\ref{d5}). 
If $a \leq d-1$, then $(m-1)a \geq (k-1)d+1 \geq 5$. In this case a p.p.d. $r_{2}$ for $(3,(m-1)a)$  
divides the LHS, but not the RHS of (\ref{d5}). If $a=d$, (\ref{d5}) yields $3^{n}-5 = 4$, a contradiction. So $a \geq d+1$. 
Now if $k=3$, then $m = 2$ and (\ref{d5}) yields 
$(3^{3d}-5)(3^{3d/2}-1) = 4(3^{d}-1)(3^{2d}-1)$, a contradiction. So $k \geq 4$.
As $(m-1)a \geq d+1 \geq 3$, there is a p.p.d. $r_{2}$ for 
$(3,(m-1)a)$. Then (\ref{d5}) implies that $jd = (m-1)ab$ for some integers $b \geq 1$, 
$1 \leq j \leq d$. Hence $jd = ma-a = kd-a$ and so $d|a$ if $b = 1$. If $b \geq 2$, 
then $kd \geq jd \geq 2(m-1)a \geq ma = kd$, whence $j=k$, $b=m=2$, and $a = kd/2 \geq 2d$.
Thus in either case $a \geq 2d$. It follows that the LHS of 
(\ref{d5}) is  $<3^{n(3+n/2d)/2}$, whereas its RHS is 
$>3^{n(1+n/d)/2}$, contradicting $n \geq 4d$.

Let $\chi(1) = N  (3^{n}-1)/4$ and $d \geq 3$. If $n = 4$ use \cite{Atlas}. 
Let $n \geq 6$. 
We have 
\begin{equation}\label{d1}
\textstyle  \prod^{t}_{i=1}\prod^{m_{i}}_{j=1}(3^{ja_{i}}-1) = 4\prod^{k-1}_{i=1}(3^{di}-1).
\end{equation}
The LHS of (\ref{d1}) is $\geq 2^{n} > 8$, so $k > 1$. 
As $(k-1)d \geq n/2 \geq 3$, there is a p.p.d. $r_{3}$ for $(3,(k-1)d)$. 
By (\ref{d1}), we may assume $m_{1}a_{1} \geq (k-1)d$. If $m_{1}a_{1} > (k-1)d$, 
then a p.p.d. for $(3,m_{1}a_{1})$ divides the LHS but not the RHS of (\ref{d1}). Hence $m_{1}a_{1} = (k-1)d$ 
and $\sum^{t}_{i=2}m_{i}a_{i} = d$. Now the LHS of (\ref{d1}) is  
$\geq 2^{d}(3^{d}-1) > 4(3^{d}-1)$, so $k > 2$. Assume $a_{1} \leq d-1$, whence 
$(m_{1}-1)a_{1} \geq (k-2)d+1 \geq 4$. Then a p.p.d. $r_{4}$ for $(3,(m_{1}-1)a_{1})$ divides the LHS, but not the 
RHS of  
\begin{equation}\label{d12}
 \textstyle (\prod^{t}_{i=1}\prod^{m_{i}}_{j=1}(3^{ja_{i}}-1))/(3^{(k-1)d}-1) = 
    4\prod^{k-2}_{i=1}(3^{di}-1).
\end{equation}
Further, if $a_{1} = d$, then 
$\prod^{t}_{i=2}\prod^{m_{i}}_{j=1}(3^{ja_{i}}-1) \geq 2^{d} > 4$, contradicting (\ref{d12}). 
Thus $a_{1} > d$. As $(k-2)d \geq d \geq 3$, there is a p.p.d. $r_{5}$ for $(3,(k-2)d)$. Now $r_{5}$ divides the 
RHS of (\ref{d12}) and $(k-2)d > (k-1)d-a_{1} = (m_{1}-1)a_{1}$, so  $t = 2$ and $(k-2)d = d = m_{2}a_{2}$. Thus $k = 3$, $2d = m_{1}a_{1}$,
and so $m_{1} = 1$ as $a_{1} > d$. In this case, (\ref{d12}) implies  
$\prod^{m_{2}-1}_{j=1}(3^{ja_{2}}-1) = 4$, which is impossible.
\end{proof}

\begin{Theorem}\label{TTM}
If $S = SL_{n}(q)$ 
and $U,V \in \IBr(S)$ are both of dimension greater
than $1$, then $U \otimes V$ is reducible. 
\end{Theorem}  

\begin{proof}
By \cite[3.2, 3.3]{MT}, we may assume $q = 3$. The case $n \leq 2$ is easy, so let $n \geq 3$.
By \cite[3.4]{MT}, we may also assume that $U$ is a Weil module, at least one of 
$U$, $V$ does not lift to  $\CC$, $2|n$, and $V$ 
does not extend to $G := GL_{n}(3)$.
There is $W \in \IBr(G)$ such that $W\dar_{S} = V_{1} \oplus V_{2}$ where $V_{i} \in \IBr(S)$ and
$V \cong V_{1}$. Note that $U$ extends to a $G$-module which we also denote by $U$.  

First we assume that $W$ is induced from a module $W_1$ over a parabolic 
$P = UL$ with Levi $L = GL_{a} \times GL_{b}$. As the Harish-Chandra induction is commutative,
we may assume $a \geq 2$. Let $Z(G) = \langle z\rangle$, 
$Z(GL_{a}) = \langle z_{a}\rangle$, $Z(GL_{b}) = \langle z_{b}\rangle$. 
Then $z = z_{a}z_{b}$. 
Let $\tau_{n}$ denote the complex permutation character of $G$ on the points of  $\NC$. Then $\tau_{n}$ gives
rise to two Weil characters of $G$: $\tau^{0}_{n}$ of degree $(3^{n}-3)/2$, and 
$\tau^{1}_{n}$ of degree $(3^{n}-1)/2$, both restricting irreducibly to $S$, see \cite{TZ1}. Moreover, 
the $1$- (resp. $-1$-) eigenspace of $z$ in $\tau_{n}$ affords the 
$G$-character $2 \cdot 1_{G} + \tau^{0}_{n}$ (resp. $\tau^{1}_{n}$). Clearly, $U\cdot GL_{b}$ fixes 
every point of the natural subspace $\FF_{3}^{a}$. So $(\tau^{i}_{n})\dar_{P}$ 
contains $\tau^{i}_{a}$, as a Weil character of $GL_{a}$ inflated to $P$. By \cite{GT1}, we may
assume that the Brauer character of $U$ is $\bar\tau^{j}_{n} - c$, 
where $\bar{\cdot}$ denotes the restriction
to $\ell'$-elements, $j \in \{0,1\}$, and $c \in \{0,1,2\}$. Now  
$3 \leq \tau^{i}_{a}(1) \leq (3^{a}-1)/2 < (3^{a+1}-5)/4 \leq (\dim U)/2$ as $a \geq 2$. Therefore $U\dar_{P}$ contains an irreducible constituent $U_{1}$ of dimension less 
than $(\dim U)/2$. Hence 
$U \otimes W = U \otimes (W_{1}) \uar^{G} \cong (U\dar_{P} \otimes W_{1}) \uar^{G}$ contains 
a subquotient $(U_{1} \otimes W_{1})\uar^{G}$ of dimension less than $(\dim U \otimes W)/2$. 
Since $V_{1}$ and $V_{2}$ are fused by $G$, it follows that $U \otimes V$ is reducible (over $S$). 

Now let $W$ be Harish-Chandra indecomposable. As the module 
$1_{P_{n-1} \cap S}\uar^{S}$ satisfies Lemma~\ref{hom-tensor}(ii) and $\dim \End_{P_{n-1} \cap S}(U) \geq 2$, by Lemma \ref{hom-tensor} we may assume   $\dim \End_{P_{n-1} \cap S}(V) = 1$ . Now Theorem \ref{hom22} yields $W \cong L(\si,(k))$ for 
some $\si$ of degree $d = n/k > 1$. Then $W$ lifts to  $W_{\CC} := \LC(\si,(k))$. By Proposition \ref{dim4}, $V$ lifts to a complex $S$-module if 
$d = 2$. If $U\otimes V$ is irreducible then so is  $U \otimes V_{2}$. We have the following two cases.

Case 1: $\dim U = (3^{n}-1)/2$ or $(3^{n}-3)/2$. Then $U$ lifts to a 
$\CC G$-module $U_{\CC}$. By \cite[3.2]{MT}, $U_{\CC} \otimes W_{\CC}$ is reducible. 
As $U \otimes V_{1,2}$ are irreducible, $U_{\CC} \otimes W_{\CC}$ is a 
sum of two irreducibles of dimension $(\dim U)(\dim W)/2$, 
contradicting Lemma \ref{degree}(i),(iii) if $d \geq 3$. If $d = 2$,
then $U$ and $V$ lift to $\CC$, a contradiction.

Case 2: $\dim U = (3^{n}-5)/2$; in particular, there is a $\CC G$-module $U_{\CC}$ such that 
its reduction modulo $\ell$ is $U+1_{G}$ (in the Grothendieck group). So over $S$, the 
reduction modulo $\ell$ of $U_{\CC} \otimes W_{\CC}$ is 
$(U \otimes V_{1}) + (U \otimes V_{2}) + V_{1} + V_{2}$, with $V_{1}$ and $V_{2}$ fused under $G$
and all summands being irreducible. Again, by \cite[3.2]{MT}, 
$U_{\CC} \otimes W_{\CC}$ is reducible. 
Hence at least one irreducible consituent of $U_{\CC} \otimes W_{\CC}$ has degree $(\dim U)(\dim W)/2$,  contradicting Lemma \ref{degree}(ii).
\end{proof}

\vspace{-1 mm}
\small
\ifx\undefined\bysame
\newcommand{\bysame}{\leavevmode\hbox to3em{\hrulefill}\,}
\fi

\end{document}